\documentclass[12pt,a4paper,reqno]{amsart} 
\usepackage{amssymb}
\usepackage{latexsym}
\usepackage{amsmath}
\usepackage{amsxtra}
\usepackage{mathrsfs}
\usepackage[X2,T1]{fontenc}
\usepackage[applemac]{inputenc}
\usepackage{cite}
\usepackage{calc}                   

\newcommand{\dir}{{\operatorname{dir}}}
\newcommand{\Char}{\operatorname{Char}}
\newcommand{\scal}[2]{\langle #1,#2\rangle}
\newcommand{\rr}[1]{\mathbf R^{#1}}

\newcommand{\nm}[2]{\Vert #1\Vert _{#2}}

\newcommand{\Op}[1]{\operatorname{Op}(#1)}

\newcommand{\sets}[2]{\{ \, #1\, ;\, #2\, \} }

\newcommand{\ep}{\varepsilon}
\newcommand{\fy}{\varphi}
\newcommand{\cdo}{\, \cdot \, }
\newcommand{\supp}{\operatorname{supp}}

\newcommand{\eabs}[1]{\langle #1\rangle}     

\newcommand{\vrum}{\vspace{0.1cm}}

\newcommand{\back}[1]{\backslash {#1}}
\newcommand{\op}{\operatorname{Op}}
\newcommand{\SG}{\operatorname{SG}}

\newcommand{\WFF}{\operatorname{WF}}

\newcommand{\WF}[1]{\operatorname{WF}^{#1}}

\newcommand{\norm}[1]{\langle#1\rangle}

\def\cB{{\mathcal B}}
\def\cC{{\mathcal C}}

\def\cF{{\mathscr F}}

\def\cS{{\mathscr S}}

\def\sS{{\mathscr S}}

\def\SG{\operatorname{SG}}

\def\mascP{\mathscr P}

\def\mascS{\mathscr S}

\newcommand{\WFgB}{\operatorname{WF}_{\cB}}

\numberwithin{equation}{section}          

\newtheorem{thm}{Theorem}
\numberwithin{thm}{section}

\newcommand{\rubrik}{}
\newtheorem{prop}[thm]{Proposition}
\newtheorem{cor}[thm]{Corollary}
\newtheorem{lemma}[thm]{Lemma}

\theoremstyle{definition}

\newtheorem{defn}[thm]{Definition}
\newtheorem{example}[thm]{Example}

\theoremstyle{remark}
\newtheorem{rem}[thm]{Remark}

\def\zz#1{\mathbf{Z}^{#1}}

\def\Ph{\mathfrak{F}}
\def\Phr{\mathfrak{F}^r}

\def\Sym#1{\mathrm{Sym}\left(#1\right)}
      
\newcommand{\semi}[1]{\langle #1 \rangle}      

\newcommand \ciliuno{\mathcal{C}}
\newcommand \esse{\mathscr{S}}
      
\author{Sandro Coriasco}

\address{Dipartimento di Matematica ``G. Peano'', Universit\`a degli Studi di Torino, Torino, Italy}

\email{sandro.coriasco@unito.it}

\author{Karoline Johansson}

\author{Joachim Toft}

\address{Department of Computer science, Physics and Mathematics,
Linn{\ae}us University, V{\"a}xj{\"o}, Sweden}

\email{joachim.toft@lnu.se}

\title{Calculus, continuity and global wave-front
properties for\\ Fourier integral operators on $\rr{d}$}

\keywords{Wave-front, Fourier Integral Operator, Banach, modulation, micro-local}

\subjclass[2010]{35A18, 35S30, 42B05, 35H10}

\begin{document}

\begin{abstract}
We illustrate the composition properties 
for an extended family of $\SG$
Fourier integral operators.
We prove continuity results for
operators in this class with respect to $L^2$
and weighted modulation spaces, 
and discuss continuity
on $\cS$, $\cS^\prime$ and on 
weighted Sobolev spaces. We 
study mapping properties of global wave-front sets under the action of
these Fourier integral operators. We extend classical results to 
more general situations. For example, there are no requirements of homogeneity for the phase functions. 
Finally, we apply our results to the study of of the propagation of singularities,
in the context of modulation spaces, for the solutions to the
Cauchy problems for the corresponding linear hyperbolic operators.
\end{abstract}

\maketitle

\setcounter{tocdepth}{1}

\tableofcontents

\section{Introduction}\label{sec0}
%

\par

In \cite{CJT2}, global wave-front sets with respect to convenient Banach or
Fr\'{e}chet spaces were introduced, and global mapping properties of
pseudo-differential operators of $\SG$-type
were established in terms of these wave-front sets (see, e.g.
\cite{Co, coriasco, CJT1, CJT2, CJT3, CoMa, CoPa, Me, Sc}).
For any such Banach or Fr\'{e}chet space $\cB$ and
tempered distribution $f$, the global wave-front set $\WFF _{\cB}(f)$ is the
union of three components $\WFF _{\cB}^m(f)$, $m=1,2,3$. The first
component (for $m=1$) describes the local wave-front set which informs
where $f$ locally fails to belong to $\cB$, as well as the directions where
the singularities (with respect to $\cB$) propagates. The second and third
components (for $m=2$ or $m=3$) informs where at infinity the growth and
oscillations of $f$ are strong enough such that $f$ fails to belong to $\cB$.
We remark that $\WFF _{\sS}^1(f)$, $\WFF _{\sS}^2(f)$ and $\WFF
_{\sS}^3(f)$ agree with $\WFF _{\sS}^\psi (f)$, $\WFF _{\sS}^e(f)$ and
$\WFF _{\sS}^{\psi e}(f)$, respectively, in \cite{CoMa}.
Note also that for admissible $\cB$, these wave-front sets give suitable
information for local and global behavior, since $f$ belongs to $\cB$
globally (locally), if and only if $\WFF _{\cB}(f)=\emptyset$ ($\WFF
_{\cB}^1(f)=\emptyset$).

\par

It is convenient to formulate mapping properties for pseudo-differential
operators of $\SG$-type in terms of $\SG$-ordered pairs $(\cB ,\cC)$,
where $\cB$ and $\cC$ should be appropriate target and image spaces
of the involved pseudo-differential operators. (Cf. \cite{CJT2}.) More
precisely, the pair $(\cB ,\cC )$ of spaces $\cB$ and $\cC$ containing
$\sS$ and contained in $\sS '$, is called $\SG$-ordered with respect to
the weight $\omega _0$ if the mappings
\begin{equation}\label{invertoperators}
\begin{alignedat}{2}
\op (a)\, :\, \cB &\to \cC ,&\quad 
\op (b)^*\, :\, \cC &\to \cB ,
\\[1ex]
\op (c)\, :\, \cB &\to \cB  &\quad \text{and}\quad
\op (c)\, :\, \cC &\to \cC
\end{alignedat}
\end{equation}
are continuous for every $a\in \SG ^{(\omega _0)}$, $b\in \SG
^{(1/\omega _0)}$ and $c\in \SG ^{0,0}$. If it is only required that the first
mapping property in \eqref{invertoperators} holds, then the pair $(\cB ,\cC )$
is called \emph{weakly $\SG$-ordered}. Here $\SG ^{(\omega )}$,
the set of all $\SG$-symbols with respect to $\omega$, belongs to an
extended family of symbol classes of $\SG$-type. We refer to
\cite{CoMa} for the definition of (also classical) $\SG$-symbols.
We notice that \eqref{invertoperators} is true also after $\op (b)$ is
replaced by its adjoint $\op (b)^*$, because $\op (\SG ^{(\omega )})^*
= \op (\SG ^{(\omega )})$.

\par

Important examples on $\SG$-ordered pairs are the Schwartz,
tempered distributions and modulation spaces. More
precisely, in \cite{CJT2} it is noticed that $(\mascS ,\mascS )$ and
$(\mascS ',\mascS ')$ are $\SG$-ordered pairs, and for any weight
$\omega$ and any modulation space $\cB$, there is a (unique)
modulation space $\cC$ such
that $(\cB ,\cC )$ is an $\SG$-ordered pair with respect to $\omega$. In
particular, the family of $\SG$-ordered pairs is broad in the sense that $\cB$
can be chosen as a Sobolev space, or, more general, as a Sobolev-Kato
space, since such spaces are special cases of modulation spaces.
Moreover, if $\SG ^{(\omega )}$ is a classical symbol class of
$\SG$-type and $\cB$ is a Sobolev-Kato space, then $\cC$ is also a
Sobolev-Kato space.

\par

For any $\SG$-ordered pairs $(\cB ,\cC )$ with respect to $\omega$, it is
proved in \cite{CJT1 , CJT2} that the wave-front sets with respect to
$\cB$ and $\cC$ posses convenient mapping properties. For example,
if $f\in \sS '$ and $a\in \SG ^{(\omega )}$, then \eqref{invertoperators}
is refined as
\begin{equation}\label{invertoperatorsWF}
\begin{aligned}
\WFF _{\cC}(\op (a)f) &\subseteq \WFF _{\cB}(f),
\\[1ex]
\text{and}\quad
\WFF _{\cC}^m(\op (a)f) &\subseteq \WFF _{\cB}^m(f),\qquad m=1,2,3 ,
\end{aligned}
\end{equation}
and that reversed inclusions are obtained by adding the set of characteristic
points to the left-hand sides in \eqref{invertoperatorsWF}. In particular,
since the set of characteristic points is empty for elliptic operators, it follows
that equalities are attained in \eqref{invertoperatorsWF} for such operators.

\par

In this paper we establish similar properties for Fourier integral
operators. More precisely, for any symbol $a$ in
$\SG ^{(\omega )}$ for some weight $\omega$, the Fourier integral
operator $\op _{\fy}(a)$ is given by
\begin{align*}
f\mapsto (\op _\fy (a)f)(x)&\equiv (2\pi )^{-d}\int _{\rr d}e^{i\fy (x,\xi )} a(x,\xi )
\widehat f(\xi )\, d\xi ,
\intertext{and its formal $L^2$-adjoint by}
f\mapsto (\op _\fy (a)^*f)(x)&\equiv (2\pi )^{-d}\iint _{\rr {2d}}
e^{i(\scal x\xi -\fy (y,\xi ))} \overline {a(y,\xi )} f(y)\, dyd\xi .
\end{align*}
The operator $\op _\fy ^*(a)=\op _\fy (a)^*$ is here
called Fourier integral operator of type II, while $\op _\fy (a)$ is called a
Fourier integral operator of type I, with phase function $\varphi$ and amplitude
(or symbol) $a$. The phase function $\fy$ should be in $\SG ^{1,1}_{1,1}$ and satisfy
\begin{equation}\label{phasecondCones}
\eabs {\fy ' _x(x,\xi )}\asymp \eabs \xi \quad \text{and}\quad
\eabs {\fy ' _\xi (x,\xi )}\asymp \eabs x.
\end{equation}
Here and in what follows, $A\asymp B$ means that $A\lesssim B$ and $B\lesssim A$,
where $A\lesssim B$ means that $A\le c\cdot B$, for a suitable constant $c>0$.
Furthermore, $\fy$ should also fulfill the usual (global) non-degeneracy condition
$$
|\det (\fy '' _{x\xi}(x,\xi ))|\ge c ,\qquad x,\xi \in \rr d,
$$
for some constant $c>0$.

\par

In Section \ref{sec3}, the notion on $\SG$-ordered
pair from \cite{CJT2} is reformulated to include such Fourier integral operators,
where the operators $\op (a)$ and $\op (b)^*$ in \eqref{invertoperators} are
replaced by $\op _{\fy}(a)$ and $\op _{\fy}(b)^*$, respectively, and takes into
account the phase-function $\varphi$.

\par

In order to establish wave-front results, similar to
\eqref{invertoperatorsWF}, it is also required that the phase functions
fulfill some further natural conditions, namely, that they preserve shapes in
certain ways near the points in the phase space $T^*\rr{d}\simeq \rr {2d}$
(see Section \ref{sec4}).
In fact, the definitions of wave-front sets of appropriate distributions
are based on the behavior in \emph{cones} of corresponding Fourier
transformations, after localizing the involved distributions near points or
along certain directions. 

\par

In order to explain our main results, let $\phi$ be the canonical
transformation of $T^*\rr d$ generated by $\fy$, 
and consider an elliptic
Fourier integral operator $\op _\fy (a)$ with amplitude $a\in\SG ^{(\omega _0)}$.
If $(\cB ,\cC )$ are (weakly) $\SG$-ordered with respect to  $\omega _0$ and
$\varphi$ (see Section \ref{sec3} for precise definitions), then, under some natural 
\textit{invariance conditions} on the weight $\omega_0$,
\begin{equation}\label{WFandFIO}
\WFF _{\cC}(\op _\fy (a)f) = \phi (\WFF _{\cB}(f)). 
\end{equation}
A similar result holds for $\op_\fy^*(a)f$, namely
\begin{equation}\label{WFandFIObis}
\WFF _{\widetilde{\cB}}(\op _\fy^* (a)f) = \phi^{-1} (\WFF _{\widetilde{\cC}}(f)), 
\end{equation}
when $\op_\fy^*(a)\colon\widetilde{\cC}\to\widetilde{\cB}$, with a
(in general, different) couple of admissible spaces $\widetilde{\cC},
\widetilde{\cB}$, and the inverse $\phi^{-1}$ of the canonical
transformation in \eqref{WFandFIO}. More generally,
by dropping the ellipticity of the amplitude functions, with
$a\in\SG^{(\omega_1)}$, $b\in\SG^{(\omega_2)}$, $(\cB_1 ,\cC_1,
\cB_2, \cC_2 )$ being $\SG$-ordered with respect to
$\omega_1$, $\omega_2$ and $\varphi$, we show that
\begin{equation}\label{WFandFIOter}
\WFF _{\cC_1}(\op _\fy (a)f) \subseteq \phi (\WFF _{\cB_1}(f))^\mathrm{con}
\end{equation}
and
\begin{equation}\label{WFandFIOquater}
\WFF _{{\cB_2}}(\op _\fy^* (b)f) \subseteq \phi^{-1} (\WFF _{{\cC_2}}(f))^\mathrm{con}, 
\end{equation}
provided that the phase function $\varphi$ additionally fulfills conditions
similar to those in H. Kumano-Go \cite{Ku}. Notice that these conditions on the phase function
are automatically satisfied by the phase functions arising from the
short-time solutions to hyperbolic Cauchy problems in the $\SG$-classical context,
see \cite{Co, coriasco2, CoPa, CoRo}. We then apply our results to
describe the propagation of singularities from the initial data to the
solution to such $\SG$-hyperbolic Cauchy problems.

\medskip

The results above are based on comprehensive investigations of algebraic and
continuity properties of the involved Fourier integral operators. A significant
part of these investigations concern compositions between Fourier integral
operators of type I or II, with pseudo-differential operators. This is performed in
\cite{CoTo2}, where it is proved that for any Fourier integral operators
$\op _\fy (a)$ and $\op _\fy ^*(b)$  with $a,b\in \SG ^{(\omega _1)}$, and some
$p\in \SG ^{(\omega_2 )}$, then, under suitable invariance conditions on the weights,
\begin{align*}
\op (p)\circ \op _\fy (a) &= \op _\fy (c_1) \mod \op (\cB _0),
\\[1ex]
\op (p)\circ \op _\fy ^*(b) &= \op _\fy ^*(c_2) \mod \op (\cB _0),
\\[1ex]
\op _\fy (a) \circ \op (p) &= \op _\fy (c_3) \mod \op (\cB _0)
\\[1ex]
\op _\fy ^*(b) \circ \op (p) &= \op _\fy ^*(c_4) \mod \op (\cB _0),
\end{align*}
for some $c_j\in \SG ^{(\omega_{0,j})}$, $j=1,\dots ,4$, and suitable weights
$\omega_{0,j}$. Here $\op (\cB _0)$
is a set of appropriate smoothing operators, depending on the 
symbols and the phase function. Furthermore,
if $a\in \SG ^{(\omega _1)}$ and $b\in \SG ^{(\omega _2)}$, then it is also
proved that $\op _\fy ^*(b)\circ \op _\fy (a)$ and $\op _\fy(a)\circ \op _\fy^*
(b)$ are equal to pseudo-differential operators $\op (c_5)$ and $\op (c_6)$,
respectively, for some $c_5,c_6\in \SG ^{(\omega _{0,j})}$, $j=5,6$.
We also present asymptotic formulae for $c_j$, $j=1,\dots ,6$,
in terms of $a$ and $b$, or of
$a$, $b$ and $p$, modulo smoothing terms. The extensions of the
calculus of $\SG$ Fourier integral operators developed in \cite{coriasco} 
to the classes $\SG^{(\omega _0)}_{r,\rho}$,
introduced and systematically used in \cite{CJT1, CJT2, CJT3}, is recalled in
Section \ref{sec2}.

\par

The formulae \eqref{WFandFIO}--\eqref{WFandFIOquater}, given by
the calculus recalled in Section \ref{sec2}, also rely on
certain asymptotic expansions in the framework of symbolic calculus of
$\SG$ pseudo-differential operators, as well as on continuity properties for
$\SG$-ordered pairs. 

The first of the above two points concerns making sense of expansions of the form
$$
a\sim \sum a_j ,
$$
in the framework of the generalised $\SG$-classes $\SG^{(\omega_0)}_{r,\rho}$.
The ideas are similar to the corresponding
properties in the usual H{\"o}rmander calculus in Section 18.1 in \cite{Ho1}. For
this reason, in \cite{CoTo} we have established properties of asymptotic
expansions for symbols classes of the form $S(m,g)$, parameterized
by the weight function $m$ and Riemannian metric $g$ on the phase
space (cf. Section 18.4 in \cite{Ho1}). Note here that any $\SG$-class is equal
to $S(m,g)$ for some choice of $m$ and $g$, and that similar
facts hold for the H{\"o}rmander classes $S^r _{\rho ,\delta}$. The results
therefore cover several situations on asymptotic expansions for
pseudo-differential operators.

\par

With respect to the second point above,
we study in Section \ref{sec3} some specific spaces which are
$\SG$-ordered or weakly $\SG$-ordered. For example, we present necessary and
sufficient conditions for the involved weight functions and parameters, in order
for Sobolev-Kato spaces, Sobolev spaces and modulation spaces should be
$\SG$-ordered or weakly $\SG$-ordered. A direct proof
of the continuity from $L^2(\rr{d})$ to itself of $\SG$ Fourier integral operators
with a uniformly bounded amplitude (that is, the amplitude is of order $0,0$, or,
equivalently, the weight $\omega$ is uniformly bounded), similar to the one given 
in \cite{coriasco}, can be found in \cite{CoTo2}. Moreover, taking advantage of
the calculus developed in \cite{CoTo2}, recalled in Section \ref{sec2} for the convenience
of the reader, and relying on results in 
\cite{CorNicRod1} and \cite{GT}, we prove that our classes of $\SG$ Fourier
integral operators are continuous between suitable couples of weighted
modulation spaces $(M^p_{(\omega_1)}(\rr{d}), M^p_{(\omega_2)}(\rr{d}))$.

\par

Finally, in Section \ref{sec4} we prove our main propagation results and illustrate
their application to Cauchy problems, for hyperbolic linear operators and first order
systems with constant multiplicities.

\subsection*{Acknowledgements}
The first author gratefully acknowledges the partial support from the
PRIN Project ``Aspetti variazionali e perturbativi nei problemi differenziali nonlineari''
(coordinator at Università degli Studi di Torino: Prof. S. Terracini)
during the development of the present paper.

\section{Preliminaries}\label{sec1}
%

\par

We begin by fixing the notation and recalling some basic concepts
which will be needed below.
In Subsections
\ref{subs:1.1}-\ref{subs:1.4} we mainly summarizes part of the contents
of Sections 2 in \cite{CJT2,CJT3,CoTo2}. Some of the results that we recall,
compared with their original formulation in the $\SG$ context
appeared in \cite{coriasco}, are here given in a slightly more general
form, adapted to the definitions given in Subsection \ref{subs:1.3}. 

\par
\subsection{Weight functions}\label{subs:1.1}
Let $\omega$ and $v$ be positive measurable functions
on $\rr d$. Then $\omega$ is called $v$-moderate if
\begin{equation}\label{moderate}
\omega (x+y) \lesssim \omega (x)v(y)
\end{equation}
If $v$ in \eqref{moderate} can be chosen as a polynomial, then $\omega$ is
called a function or weight of \emph{polynomial type}.
We let $\mathscr P(\rr d)$ be the set
of all polynomial type functions on $\rr d$. If $\omega (x,\xi
)\in \mathscr P(\rr {2d})$ is constant with respect to the
$x$-variable or the $\xi$-variable, then we sometimes write $\omega
(\xi )$, respectively $\omega (x)$, instead of $\omega (x,\xi )$,
and consider
$\omega$ as an element in $\mathscr P(\rr {2d})$ or in $\mathscr P(\rr
d)$ depending on the situation. We say that $v$ is submultiplicative
if \eqref{moderate} holds for $\omega=v$. For convenience we assume
that all submultiplicative weights are even, and
$v$ and $v_j$ always stand for submultiplicative weights, if nothing else is stated.

\par

Without loss of generality we may assume that every $\omega \in
\mathscr P(\rr d)$ is smooth and satisfies the ellipticity condition
$\partial ^\alpha \omega / \omega \in L^\infty$. In fact,
by Lemma 1.2 in \cite {To8} it follows
that for each $\omega \in \mathscr P(\rr d)$, there is a smooth and
elliptic $\omega _0\in \mathscr P(\rr d)$ which is equivalent to
$\omega$ in the sense
\begin{equation}\label{onestar}
\omega \asymp \omega _0.
\end{equation}

\par

The weights involved in the sequel have to satisfy additional conditions.
More precisely let $r,\rho \ge0$. Then $\mathscr P_{r,\rho}(\rr {2d})$ is
the set of all $\omega (x,\xi )$ in $\mathscr P(\rr {2d})\bigcap$ $C^\infty (\rr
{2d})$ such that
\begin{equation}\label{SGCondWeight}
\eabs x^{r|\alpha |}\eabs \xi ^{\rho |\beta |}\frac {\partial ^\alpha
_x\partial ^\beta _\xi \omega (x,\xi )}{\omega (x,\xi )}\in L^\infty
(\rr {2d}),
\end{equation}
for every multi-indices $\alpha$ and $\beta$. Any weight
$\omega\in\mascP_{r,\rho}(\rr {2d})$ is then called $\SG$-moderate on
$\rr{2d}$, of order $r$ and $\rho$. Notice that $\mathscr P_{r,\rho}$
is different here compared to \cite{CJT1}, and
there are elements in $ \mathscr P(\rr {2d})$ which have
no equivalent elements in $\mathscr P_{r,\rho}(\rr {2d})$.
On the other hand, if $s, t\in
\mathbf R$ and $r, \rho \in [0,1]$, then $\mathscr P_{r,\rho} (\rr {2d})$
contains all weights of the form
\begin{equation}\label{varthetaWeights}
\vartheta _{m,\mu} (x,\xi )\equiv \eabs x ^m\eabs \xi ^\mu ,
\end{equation}
which are one of the most common type of weights.

\par

It will also be useful to consider $\SG$-moderate weights in one or three
sets of variables. Let $\omega \in \mathscr P(\rr {3d})\bigcap C^\infty (\rr {3d})$,
and let $r_1,r_2,\rho \ge 0$. Then $\omega$ is called $\SG$ moderate on $\rr{3d}$,
of order $r_1$, $r_2$ and $\rho$, if  it fulfills
$$
\eabs {x_1}^{r_1|\alpha _1|} \eabs {x_2}^{r_2|\alpha _2|} \eabs \xi ^{\rho |\beta |}
\frac{\partial _{x_1}^{\alpha _1} \partial _{x_2}^{\alpha _2} \partial _\xi ^\beta \omega (x_1,x_2,\xi )}
       {\omega (x_1,x_2,\xi )}
\in L^\infty(\rr {3d}).
$$
The set of all $\SG$-moderate weights on $\rr{3d}$ of order $r_1$, $r_2$ and $\rho$
is denoted by $\mascP _{r_1,r_2,\rho}(\rr{3d})$. Finally, we denote by $\mascP_{r}(\rr{d})$ the
set of all $\SG$-moderate weights of order $r\ge0$ on $\rr d$, which are defined in a
similar fashion.

\par
\subsection{Modulation spaces}\label{subs:1.2}
Let $\phi \in \mathscr S(\rr d)$. Then the \emph{short-time Fourier transform} of $f\in \mathscr
S(\rr d)$ with respect to (the window function) $\phi$ is defined by
\begin{equation}\label{stftformula}
V_\phi f(x,\xi ) = (2\pi )^{-d/2}\int _{\rr {d}} f(y)\overline {\phi
(y-x)}e^{-i\scal y\xi}\, dy.
\end{equation}
More generallly, the short-time Fourier transform of $f \in \mathscr S'(\rr d)$ with
respect to $\phi  \in \mathscr S'(\rr d)$ is defined by
\begin{equation}\tag*{(\ref{stftformula})$'$}
(V_\phi f) =\mathscr F_2F,\quad \text{where}\quad F(x,y)=(f\otimes
\overline \phi)(y,y-x).
\end{equation}
Here $\mathscr F_2F$ is the partial Fourier transform of
$F(x,y)\in \mathscr S'(\rr{2d})$ with respect to the $y$-variable.
We refer to \cite{Fo, Gro-book} for more facts about the short-time Fourier
transform.
To introduce the modulation spaces, we first recall that a Banach space $\mathscr B$,
continuously embedded in $L^1_{\mathrm{loc}}(\rr d)$,
is called a \emph{(translation) invariant
BF-space on $\rr d$}, with respect to a submultiplicative weight
$v \in \mathscr P(\rr {d})$, if there is a constant $C$
such that the following conditions are fulfilled:
\begin{enumerate}
\item $\mathscr S(\rr d)\subseteq \mathscr
B\subseteq \mathscr S'(\rr d)$ (continuous embeddings);

\vrum

\item if $x\in \rr d$ and $f\in \mathscr B$, then $f(\cdot -x)\in
\mathscr B$, and
\begin{equation}\label{translmultprop1}
\nm {f(\cdot-x)}{\mathscr B}\le Cv(x)\nm {f}{\mathscr B}\text ;
\end{equation}

\vrum

\item if  $f,g\in L^1_{\mathrm{loc}}(\rr d)$ satisfy $g\in \mathscr B$
and $|f| \le |g|$ almost everywhere, then $f\in \mathscr B$ and
$$
\nm f{\mathscr B}\le C\nm g{\mathscr B}\text ;
$$

\vrum

\item if $f\in \mathscr B$ and $\fy\in C^\infty_0(\rr d)$, then $f*\fy\in\mathscr B$, and 
\begin{equation}\label{BFconvEst}
\nm{f*\fy}{\mathscr B}\leq \nm{\fy}{L^1_{(v)}} \nm{f}{\mathscr B}.
\end{equation}
\end{enumerate}

The following definition of modulation spaces is due to Feichtinger \cite{Feichtinger6}.
Let $\mathscr{B}$ be a translation invariant BF-space on
$\rr {2d}$ with respect to $v\in \mathscr P(\rr {2d})$,
$\phi \in \mathscr{S}(\rr{d})\back{0}$ and let $\omega \in
\mathscr{P}(\rr{2d})$ be such that $\omega$ is $v$-moderate.
The \emph{modulation space} $M(\omega ,\mathscr B)$ consists of all $f\in
\mathscr{S}'(\rr{d})$ such that $V_{\phi}f\cdot \omega \in
\mathscr{B}$. We notice that $M(\omega ,\mathscr B)$ is a Banach space
with the norm
\begin{equation}\label{modnorm}
\|f\|_{M(\omega ,\mathscr B)} \equiv \|V_{\phi} f
\omega\|_{\mathscr{B}}
\end{equation}
(cf. \cite{Feichtinger3}).

\par

\begin{rem}
Assume that $p,q\in [1,\infty]$, and let $L^{p,q}_{1}(\rr{2d})$ and
$L^{p,q}_{2}(\rr{2d})$ be the sets of all $F\in  L^1_{\mathrm{loc}}
(\rr{2d})$ such that
\begin{equation*}
\|F\|_{L^{p,q}_1}  \equiv \Big ( \int \Big( \int |F(x,\xi)|^p\,
dx\Big )^{q/p}\,d\xi \Big )^{1/q}
<\infty
\end{equation*}
and
\begin{equation*}
\|F\|_{L^{p,q}_2} \equiv \Big ( \int \Big ( \int |F(x,\xi)|^q\,
d\xi \Big )^{p/q}\, dx\Big )^{1/p}<\infty .
\end{equation*}
Then $M(\omega ,L^{p,q}_1(\rr{2d}))$ is equal to the \emph{classical}
modulation
space $M^{p,q}_{(\omega)}(\rr{d})$, and $M(\omega ,L^{p,q}_2(\rr{2d}))$ is
equal to the space $W^{p,q}_{(\omega)}(\rr{d})$, related to
Wiener-amalgam spaces (cf. \cite{F1,Feichtinger6,Feichtinger3,Gro-book}).
We
set $M^p_{(\omega )} = M^{p,p}_{(\omega
)}= W^{p,p}_{(\omega )}$. Furthermore, if $\omega =1$, then we write
$M^{p,q}$, $M^p$ and $W^{p,q}$ instead of $M^{p,q}_{(\omega )}$,
$M^p_{(\omega )}$ and $W^{p,q}_{(\omega )}$ respectively.
\end{rem}

\par

\begin{rem}\label{modspaceEx}
Several important spaces agree with certain modulation spaces. In fact, let
$s,\sigma \in \mathbf R$. If $\omega =\vartheta_{s,\sigma}$ (cf.
\eqref{varthetaWeights}), then $M^2_{(\omega )}(\rr d)$ is equal to the weighted
Sobolev space (or Sobolev-Kato space) $H^2_{\sigma ,s}(\rr d)$ in \cite{CoMa,Me},
the set of all $f\in \mathscr S'(\rr d)$ such that $\eabs x^s\eabs D^\sigma
f\in L^2(\rr d)$. In particular, if $s=0$ ($\sigma =0$), then
$M^2_{(\omega )}(\rr d)$ equals to $H^2_\sigma (\rr d)$ ($L^2_s(\rr d)$).
Furthermore, if instead $\omega (x,\xi )=\eabs {x,\xi }^s$, then
$M^2_{(\omega )}(\rr d)$ is equal to the
Sobolev-Shubin space of order $s$. (Cf. e.{\,}g.  \cite{LuRa}).
\end{rem}

\par

\subsection{Pseudo-differential operators and $\SG$ symbol classes}\label{subs:1.3}
Let $a\in
\mathscr S(\rr {2d})$, and $t\in \mathbf R$ be fixed. Then
the pseudo-differential operator $\op _t(a)$ is the linear and
continuous operator on $\mathscr S(\rr d)$ defined by the formula
\begin{equation}\label{e0.5}
(\op _t(a)f)(x)
=
(2\pi ) ^{-d}\iint e^{i\scal {x-y}\xi } a((1-t)x+ty,\xi )f(y)\,
dyd\xi 
\end{equation}
(cf. Chapter XVIII in \cite {Ho1}). For
general $a\in \mathscr S'(\rr {2d})$, the pseudo-differential
operator $\op _t(a)$ is defined as the continuous operator from
$\mathscr S(\rr d)$ to $\mathscr S'(\rr d)$ with distribution
kernel
\begin{equation}\label{weylkernel}
K_{t,a}(x,y)=(2\pi )^{-d/2}(\mathscr F_2^{-1}a)((1-t)x+ty,x-y).
\end{equation}
If  $t=0$, then $\op _t(a)$ is the Kohn-Nirenberg
representation $\op (a)=a(x,D)$, and if $t=1/2$, then $\op _t(a)$ is
the Weyl quantization. 

\par

In most of our situations, $a$ belongs to a generalized $\SG$-symbol
class, which we shall consider now.
Let $m,\mu ,r, \rho \in \mathbf R$
be fixed. Then the $\SG$-class $\SG
^{m,\mu }_{r,\rho}(\rr {2d})$ is the set of all $a\in
C^\infty (\rr {2d})$ such that
$$
|D _x^\alpha D _\xi ^\beta a(x,\xi )|\lesssim
\eabs x^{m-r|\alpha|}\eabs \xi ^{\mu -\rho |\beta
|},
$$
for all multi-indices $\alpha$ and $\beta$. Usually we assume that
$r,\rho \ge 0$ and $\rho +r >0$.

\par

More generally, assume that $\omega \in \mathscr P_{r ,\rho}
(\rr {2d})$. Then $\SG _{r,\rho}^{(\omega )}(\rr {2d})$ consists of all $a\in
C^\infty (\rr {2d})$ such that
\begin{equation}\label{Somegadef}
	|D_x^\alpha D_\xi ^\beta a(x,\xi)|\lesssim \omega(x,\xi)\eabs
        x^{-r|\alpha|}\eabs \xi ^{-\rho|\beta|}, \qquad
        x,\xi\in \rr d,
\end{equation}
for all multi-indices $\alpha$ and $\beta$. We notice that
\begin{equation}\label{SG}
\SG _{r,\rho}^{(\omega )}(\rr {2d})=S(\omega ,g_{r,\rho}
),
\end{equation}
when $g=g_{r,\rho}$ is the Riemannian metric on $\rr {2d}$,
defined by the formula
\begin{equation}\label{riemannianmetric}
\big (g_{r,\rho}\big )_{(y,\eta )}(x,\xi ) =\eabs y ^{-2r}|x|^2 +\eabs
\eta ^{-2\rho}|\xi |^2
\end{equation}
(cf. Section 18.4--18.6 in \cite{Ho1}). Furthermore, $\SG ^{(\omega
)}_{r,\rho} =\SG ^{m,\mu }_{r,\rho}$ when $\omega =\vartheta _{m,\mu}$
(see \eqref{varthetaWeights}).

\par

For conveniency we set
\begin{gather*}
\SG ^{(\omega \vartheta _{-\infty ,0} )} _\rho (\rr {2d})
= \SG ^{(\omega \vartheta _{-\infty ,0} )} _{r,\rho} (\rr {2d})
\equiv
\bigcap _{N\ge 0} \SG ^{(\omega \vartheta _{-N ,0} )} _{r,\rho} (\rr {2d}),
\\[1ex]
\SG ^{(\omega \vartheta _{0,-\infty } )} _r (\rr {2d})
= \SG ^{(\omega \vartheta _{0,-\infty } )} _{r,\rho} (\rr {2d})
\equiv
\bigcap _{N\ge 0} \SG ^{(\omega \vartheta _{0,-N } )} _{r,\rho} (\rr {2d}),
\intertext{and}
\SG ^{(\omega \vartheta _{-\infty ,-\infty } )}  (\rr {2d})
= \SG ^{(\omega \vartheta _{-\infty ,-\infty } )} _{r,\rho} (\rr {2d})
\equiv
\bigcap _{N\ge 0} \SG ^{(\omega \vartheta _{-N ,-N} )} _{r,\rho} (\rr {2d}).
\end{gather*}
We observe that $\SG ^{(\omega \vartheta _{-\infty ,0} )} _{r,\rho} (\rr {2d})$
is independent of $r$, $\SG ^{(\omega \vartheta _{0,-\infty } )} _{r,\rho} (\rr {2d})$
is independent of $\rho$, and that
$\SG ^{(\omega \vartheta _{-\infty ,-\infty } )} _{r,\rho} (\rr {2d})$ is independent
of both $r$ and $\rho$. Furthermore, for any $x_0,\xi _0\in \rr d$ we have
\begin{alignat*}{3}
\SG ^{(\omega \vartheta _{-\infty ,0} )} _\rho (\rr {2d}) &=
\SG ^{(\omega _0\vartheta _{-\infty ,0} )} _\rho (\rr {2d}),&
\quad &\text{when}& \quad \omega _0(\xi ) &= \omega (x_0,\xi ),
\\[1ex]
\SG ^{(\omega \vartheta _{0,-\infty } )} _r (\rr {2d}) &=
\SG ^{(\omega _0\vartheta _{0,-\infty } )} _r (\rr {2d}),&
\quad &\text{when}& \quad \omega _0(x ) &= \omega (x,\xi _0),
\intertext{and}
\SG ^{(\omega \vartheta _{-\infty ,-\infty } )}  (\rr {2d}) &=
\mathscr S(\rr {2d}).& & & &
\end{alignat*}

\par

The following result shows that
the concept of asymptotic expansion extends to the classes
$\SG^{(\omega)}_{r,\rho}(\rr{2d})$. We refer to
\cite[Theorem 8]{CoTo} for the proof.

\par

\begin{prop}\label{propasymp}
Let $r,\rho \ge 0$ satisfy $r+\rho >0$, and let
$\{ s_j \} _{j\ge 0}$ and $\{ \sigma _j \} _{j\ge 0}$ be sequences of
non-positive numbers
such that $\lim _{j\to \infty} s_j =-\infty$ when $r>0$ and $s_j=0$ otherwise,
and $\lim _{j\to \infty} \sigma _j =-\infty$ when $\rho >0$ and $\sigma _j=0$
otherwise. Also let $a_j\in\SG ^{(\omega_j)}_{r,\rho}(\rr{2d})$, $j=0,1,\dots$,
where $\omega_j=\omega \cdot \vartheta_{s_j,\sigma_j}$. Then there is a
symbol $a\in\SG^{(\omega)}_{r,\rho}(\rr{2d})$ such that
\begin{equation}
	\label{eq:sgassum}
	a-\sum_{j=0}^Na_j\in\SG^{(\omega_{N+1})}_{r,\rho}(\rr{2d}).
\end{equation}
The symbol $a$ is uniquely determined modulo a remainder $h$, where
\begin{equation}
	\label{eq:gensgasexp}
	\begin{alignedat}{3}
		h &\in \SG ^{\omega \vartheta _{-\infty ,0})}_{\rho}(\rr{2d}) & \quad
		&\text{when}& \quad r&>0,
		\\[1ex]
		h &\in \SG^{(\omega \vartheta _{0,-\infty} )}_{r}(\rr{2d}) & \quad
		&\text{when} & \quad \rho &> 0,
		\\[1ex]  
		h &\in \mathscr S(\rr{2d}) & \quad &\text{when} &\quad r &> 0,
		\rho >0.
	\end{alignedat}
\end{equation}
\end{prop}

\par

\begin{defn}\label{def:gensgasexp}
The notation $a\sim \sum a_j$ is used when $a$ and $a_j$
fulfill the hypothesis in Proposition \ref{propasymp}. Furthermore,
the formal sum
$$
\sum_{j\ge0}a_j
$$
is called an \emph{asymptotic expansion}.
\end{defn}

\par

It is a well-known fact that $\SG$-operators give rise to linear continuous mappings 
from $\mathscr S(\rr d)$ to itself, extendable as linear continuous mappings
from $\mathscr S'(\rr d)$ to itself. They also act continuously between modulation
spaces. Indeed,  if $a\in \SG ^{(\omega _0)}_{r,\rho} (\rr {2d})$, then $\op _t(a)$
is continuous from $M(\omega ,\mathscr B)$ to $M(\omega /\omega
_0,\mathscr B)$ (cf. \cite{CJT2}). Moreover, 
there exist $a\in \SG ^{(\omega _0)}_{r,\rho} (\rr {2d})$ and $b\in \SG
^{(1/\omega _0)}_{r,\rho} (\rr {2d})$ such that for every choice of $\omega
\in \mathscr P(\rr {2d})$ and every translation invariant BF-space $\mathscr
B$ on $\rr {2d}$, the mappings
\begin{gather*}
\op _t(a)\, : \, \mathscr S(\rr d)\to \mathscr S(\rr d),\quad \op _t(a)\, : \,
\mathscr S'(\rr d)\to \mathscr S'(\rr d)
\\[1ex]
\text{and}\quad \op _t(a)\, : \, M(\omega ,\mathscr B)\to M(\omega
/\omega _0,\mathscr B).
\end{gather*}
are continuous bijections with inverses $\op _t(b)$.

\par

\subsection{Composition and further properties of $\SG$ classes of
symbols, amplitudes, and functions}\label{subs:1.4}
We define families of \emph{smooth functions with $\SG$ behaviour}, depending on
one, two or three sets of real variables (cfr. also \cite{CoSch}). We then introduce
pseudo-differential operators defined by means of $\SG$ amplitudes. Subsequently,
we recall sufficient conditions for maps of $\rr{d}$ into itself
to keep the invariance of the $\SG$ classes. 

%
%

\par

In analogy of $\SG$ amplitudes defined on $\rr {2d}$, we consider
corresponding classes of amplitudes defined on $\rr {3d}$. More precisely,
for any $m_1, m_2, \mu, r_1,r_2,\rho \in \mathbf R$,
let $\SG ^{m_1,m_2,\mu }_{r_1,r_2,\rho} (\rr{3n})$ be the set of all
$a \in C^\infty \left( \rr{3d} \right )$ such that
\begin{equation}
\label{eq:1.1}       
                  |\partial ^{\alpha _1}_{x_1}
                  \partial ^{\alpha _2}_{x_2}
                   \partial _\xi ^\beta
                  a(x_1, x_2, \xi)|               
                 \lesssim
                   \eabs{x_1}^{m_1 - r_1|\alpha _1|}
                  \eabs{x_2}^{m_2 - r_2|\alpha _2|} \eabs{\xi}^{\mu  - \rho |\beta |},
\end{equation}
for every multi-indices $ \alpha_1, \alpha_2, \beta$. We usually assume
$r_1,r_2,\rho \ge 0$ and $r_1+r_2+\rho>0$. More generally,  let
$\omega \in \mascP _{r_1,r_2,\rho}(\rr{3d})$. Then
$\SG ^{(\omega )}_{r_1,r_2,\rho} (\rr{3d})$ is the  set of all
$a \in C^\infty \left( \rr{3d} \right)$ which satisfy
\begin{equation}
\tag*{(\ref{eq:1.1})$'$}
                   |\partial ^{\alpha _1}_{x_1}
                  \partial ^{\alpha _2}_{x_2}
                   \partial _\xi ^\beta
                  a(x, y, \xi)|               
                 \lesssim
                   \omega(x_1,x_2,\xi )\eabs{x_1}^{- r_1|\alpha _1|}
                  \eabs{x_2}^{- r_2|\alpha _2|} \eabs{\xi}^{- \rho |\beta |},
\end{equation}
for every multi-indices $ \alpha_1, \alpha_2, \beta$. The set
$\SG ^{ (\omega) }_{r_1,r_2,\rho} (\rr{3n})$ is equipped with
the usual Fr\'{e}chet topology based upon the seminorms
implicit in \eqref{eq:1.1}$'$. 

\par

As above,
$$
\SG ^{(\omega)}_{r_1,r_2,\rho} =\SG ^{m_1,m_2,\mu }_{r_1,r_2,\rho}
\quad \text{when}\quad
\omega (x_1,x_2,\xi )=\eabs
{x_1}^{m_1}\eabs {x_2}^{m_2}\eabs \xi ^\mu .
$$

\par

\begin{defn}
\label{def:psidoamp}
Let $r_1,r_2,\rho\ge0$, $r_1+r_2+\rho>0$, and
let $a\in \SG^{(\omega)}_{r_1,r_2,\rho}(\rr{3d})$, where $\omega \in
\mascP_{r_1,r_2,\rho}(\rr{3d})$. Then, the pseudo-differential operator
$\op(a)$ is the linear and continuous operator from $\cS(\rr d)$ to
$\cS^\prime(\rr d)$ with distribution kernel
\[
	K_{a}(x,y)=(2\pi )^{-d/2}(\mathscr F_3^{-1}a)(x,y,x-y).
\]
For $f\in\cS(\rr d)$, we have
\[
	(\op(a)f)(x)=(2\pi)^{-d}\iint e^{i \scal {x-y}\xi} a(x,y,\xi)f(y)\,dyd\xi.
\]
\end{defn}
The operators introduced in Definition \ref{def:psidoamp} have properties
analogous to the usual $\SG$ operator families described in \cite{Co}.
They coincide with the operators defined in the previous subsection,
where corresponding symbols are obtained by means of asymptotic
expansions, modulo remainders of
the type given in \eqref{def:gensgasexp}.
For the sake of brevity, we here omit the details. Evidently, when
neither the amplitude functions $a$, nor the corresponding weight $\omega$,
depend on $x_2$, we obtain the definition of
$\SG$ symbols and pseudo-differential operators, given in the
previous subsection.

\par

Next we consider \emph{$\SG$ functions}, also called
\emph{functions with $\SG$ behavior}. That is, amplitudes which
depend only on one set of variables in $\rr d$. We denote them
by $\SG ^{(\omega )}_r(\rr d)$ and $\SG^{m}_r(\rr d)$, $r>0$,
respectively, for a general weight $\omega\in\mascP_{r}(\rr{d})$
and for $\omega(x)=\eabs x^m$. Furthermore, if
$\phi \colon \rr {d_1} \to \rr {d_2}$,
and each component $\phi _j$, $j=1, \dots, d_2$, of $\phi$ belongs to
$\SG^{(\omega)}_r(\rr {d_1})$, we will occasionally write $\phi \in
\SG^{(\omega)} _r(\rr {d_1};\rr {d_2})$. We use similar notation
also for other vector-valued $\SG$ symbols and amplitudes.

\par

In the sequel we will need to consider compositions of $\SG$
amplitudes with functions with $\SG$ behavior. In particular,
the latter will often be $\SG$ maps (or diffeomorphisms) with
$\SG^0$-parameter dependence, generated by phase
functions (introduced in \cite{coriasco}), 
see Definitions \ref{def:sgmap} and \ref{def:sgmap}, and
Subsection \ref{subs:2.1} below. For the convenience of the
reader, we first recall, in a form slightly more general than the 
one adopted in \cite{coriasco}, the definition $\SG$
diffeomorphisms with $\SG^0$-parameter dependence.

\par

\begin{defn}\label{def:sgmap}
Let $\Omega _j \subseteq \rr{d_j}$ be open, $\Omega =
\Omega _1\times \cdots \times \Omega _k$ and let
$\phi \in C^\infty (\rr d\times \Omega ;\rr d)$. Then $\phi$ is
called an $\SG$ map (with $\SG
^0$-parameter dependence) when the following conditions hold:
\begin{enumerate}
	\item \label{cond:1}
	$\norm{\phi(x,\eta )}\asymp \norm {x}$, 
	uniformly with respect to $\eta \in \Omega$;
	\item for all $\alpha \in \zz {d}_+$, $\beta =
	(\beta _1,\dots ,\beta _k)$, $\beta _j \in \zz{d_j}_+$,
	$j=1,\dots, k$,
	and any $(x,\eta )\in \rr{d} \times \Omega$,
	\[
		|\partial ^\alpha _x\partial ^{\beta _1}_{\eta _1}\cdots  \partial
		^{\beta _k}_{\eta _k}\phi(x,\eta )|
		\lesssim
		\norm{x}^{1-|\alpha |}\norm{\eta _1}^{-|\beta _1|}\cdots
		\norm{\eta _k}^{-|\beta _k|},
	\]
where $\eta =(\eta _1,\dots ,\eta _k)$ and $\eta _j\in \Omega _j$ for every
$j$.
\end{enumerate}
\end{defn}

\par

\begin{defn}\label{def:sgdiffeo}
Let $\phi \in C^\infty (\rr{d}\times \Omega ;\rr d )$ be an $\SG$
map. Then $\phi$ is called an $\SG$
diffeomorphism (with $\SG^0$-parameter dependence) when
there is a constant $\varepsilon>0$ such that
\begin{equation}\label{eq:sgreg}
	|\det \phi ^\prime _x(x,\eta )|\ge\varepsilon,
\end{equation}
uniformly with respect to $\eta \in \Omega$.
\end{defn}
\begin{rem}
	Condition (\ref{cond:1}) in Definition \ref{def:sgmap} and
	\eqref{eq:sgreg}, together with abstract results
	(see, e.g., \cite{Berger}, page 221) and the inverse function
	theorem, imply that,
	for any $\eta \in \Omega$, an $\SG$ diffeomorphism
	$\phi(\cdo ,\eta )$ is a smooth, global bijection from
	$\rr {d}$ to itself with smooth inverse $\psi(\cdo ,\eta )
	=\phi ^{-1}(\cdo ,\eta )$.
	It can be proved that also the inverse mapping
	$\psi (y,\eta )=\phi ^{-1}(y,\eta )$ fulfills Conditions
	(1) and (2) in Definition \ref{def:sgmap}, 
	as well as \eqref{eq:sgreg}, see \cite{coriasco}.
\end{rem}

\par

\begin{defn}\label{def:omegainv}
	Let $r,\rho \ge 0$, $r+\rho >0$, $\omega \in \mascP_{r,\rho}
	(\rr{2d})$, and let $\phi ,\phi _1,\phi _2\in C^\infty (\rr{d}\times
	\rr{d_0};\rr d)$
	be $\SG$ mappings.
	\begin{enumerate}
		\item $\omega$ is called \emph{$(\phi,1)$-invariant}
		when
		\[
			\omega(\phi(x,\eta _1+\eta _2),\xi )\lesssim
			\omega
			(\phi(x,\eta _1),\xi ),
		\]
		for any $x, \xi \in \rr{d}$, $\eta _1,\eta _2\in \rr{d_0}$, uniformly with
		respect to $\eta _2\in \rr{d_0}$. The set of all $(\phi,1)$-invariant
		weights in $\mascP _{r,\rho}(\rr{2d})$ is denoted by
		$\mascP_{r,\rho}^{\phi,1}(\rr{2d})$;
		\item $\omega$ is called \emph{$(\phi,2)$-invariant}
		when
		\[
			\omega(x,\phi(\xi ,\eta _1+\eta _2))\lesssim
			\omega
			(x,\phi(\xi ,\eta _1)),
		\]
		for any $x, \xi \in \rr{d}$, $\eta _1,\eta _2\in \rr{d_0}$, uniformly with
		respect to $\eta _2\in \rr{d_0}$. The set of all $(\phi,2)$-invariant
		weights in $\mascP _{r,\rho}(\rr{2d})$ is denoted by
		$\mascP_{r,\rho}^{\phi,2}(\rr{2d})$;
		\item $\omega$ is called \emph{$(\phi _1,\phi _2)$-invariant}
		if $\omega$ is both $(\phi _1,1)$-invariant and $(\phi _2,2)$-invariant.
		The set of all $(\phi _1,\phi _2)$-invariant
		weights in $\mascP _{r,\rho}(\rr{2d})$ is denoted by
		$\mascP_{r,\rho}^{(\phi_1,\phi _2)}(\rr{2d})$
	\end{enumerate}
\end{defn}

\par

The next Lemma \ref{lemma:omegainv}, proved in \cite{CoTo2}, shows that, under mild additional conditions, 
the families of weights introduced in Subsection 
\ref{subs:1.1} are indeed ``invariant'' under composition
with $\SG$ maps with $\SG^0$-parameter dependence.
That is, the compositions introduced in Definition
\ref{def:omegainv} are still weight functions in the sense
of Subsection \ref{subs:1.1}, belonging to suitable sets
$\mathscr{P}_{r,\rho}(\rr{2d})$.

\par

\begin{lemma}\label{lemma:omegainv}
	Let $r,\rho\in[0,1]$, $r+\rho>0$, $\omega\in\mascP
	_{r,\rho}(\rr{2d})$, and let 
	$\phi\colon\rr{d}\times\rr{d}\to\rr{d}$
	be an $\SG$ map as in Definition \ref{def:sgmap}.
	The following statements hold true.
	\begin{enumerate}
		\item Assume $\omega\in\mascP_{1,\rho}^{\phi,1}(\rr{2d})$, 
		and set $\omega_1(x,\xi):=\omega(\phi(x,\xi),\xi)$.
		Then, $\omega_1\in\mascP_{1,\rho}(\rr{2d})$.
		\item Assume $\omega\in\mascP_{r,1}^{\phi,2}(\rr{2d})$, 
		and set $\omega_2(x,\xi):=\omega(x,\phi(\xi,x))$.
		Then, $\omega_2\in\mascP_{r,1}(\rr{2d})$.
	\end{enumerate}
\end{lemma}
\begin{rem}
\label{rem:3.2}
It is obvious that, when dealing with Fourier integral operators,
the requirements for $\phi$ and $\omega$ in Lemma \ref{lemma:omegainv}
need to be  satisfied only on the support of the involved amplitude.
By Lemma \ref{lemma:omegainv}, 
it also follows that if $a \in \SG^{(\omega)}_{1,1}(\rr{2d})$ and
$\phi=(\phi_1,\phi_2)$, where
$\phi_1\in  \SG^{1,0}_{1,1}(\rr{2d})$ and $\phi_2\in  \SG^{0,1}_{1,1}(\rr{2d})$
are $\SG$ maps with $\SG^0$ parameter dependence, 
then $a\circ \phi \in \SG^{({\omega _0})}_{1,1}(\rr{2d})$ when
${\omega _0}:=\omega \circ \phi$,
provided $\omega$ is $(\phi _1,\phi _2)$-invariant. Similar results
hold for $\SG$ amplitudes and weights defined on $\rr{3d}$.
\end{rem}

\par

\begin{rem}\label{rem:trivweight}
	By the definitions it follows that any weight
	$\omega=\vartheta_{s,\sigma}$, $s,\sigma\in\mathbf R$,
	 is $(\phi,1)$-,
	$(\phi,2)$-, and $(\phi_1,\phi_2)$-invariant with respect to any
	$\SG$ diffeomorphism 
	with $\SG^0$ parameter dependence $\phi$, $(\phi_1,\phi_2)$. 
\end{rem}

\par

\section{Symbolic calculus for generalised FIOs of $\SG$ type}
\label{sec2}
%

We here recall the class of Fourier integral operators we are interested in,
generalizing those studied in \cite{coriasco}. The corresponding symbolic
calculus has been obtained in \cite{CoTo2}, from which we recall the results
listed below, and to which we refer the reader for the details. A key tool in the proofs of the composition theorems below are the results on asymptotic expansions in the Weyl-H\"ormander calculus obtained in \cite{CoTo}.

%
%
%

\subsection{Phase functions of $\SG$ type}\label{subs:2.1}
We recall the definition of the class of admissible phase functions in 
the $\SG$ context, as it was given in \cite{coriasco}. We then observe
that the subclass of \emph{regular phase functions} generates
(parameter-dependent) mappings of $\rr d$ onto itself, which turn
out to be $\SG$ maps with $\SG^0$ parameter-dependence.
Finally, we define some \emph{regularizing operators},
which are used to prove the properties of the $\SG$ Fourier integral
operators introduced in the next subsection.

\par

\begin{defn}\label{def:2.1}
A real-valued function $\varphi \in \SG^{1,1}_{1,1}(\rr {2d})$
is called a \emph{simple phase function} (or \emph{simple phase}), if
\begin{equation}
\label{eq:2.0}
\norm{\varphi_{\xi}^\prime(x,\xi)} \asymp \norm{x}
       \mbox{ and }   
 \norm{\varphi_{x}^\prime(x,\xi)} \asymp \norm{\xi},
\end{equation} 
are fulfilled, uniformly with respect to $\xi$ and $x$, repectively. 
The set of all simple phase functions is denoted by $\Ph$.
Moreover, the simple phase function $\varphi$ is called \emph{regular},
if $\left|\det (\varphi^{\prime\prime}_{x \xi} (x,\xi) ) \right| \ge c$ for some $c>0$
and all $x,\xi \in \rr d$. The set of all regular phases is denoted by $\Phr$.
\end{defn}

\par

We observe that a regular phase function $\varphi$
defines two globally invertible mappings,
namely $\xi \mapsto  \varphi^\prime_x(x,\xi)$ and $x \mapsto  \varphi
^\prime_\xi (x,\xi)$, see the analysis in \cite{coriasco}. 
Then, the following result holds true for the mappings $\phi_1$ and $\phi_2$
generated by the first derivatives of the admissible regular phase functions.

\par

\begin{prop}
\label{prop:3.2}
Let $\varphi \in \Ph$. Then, for any $x_0,\xi _0\in \rr d$,
$\phi_1\colon\rr{d}\to\rr{d}\colon x \mapsto  \varphi^\prime_\xi (x,\xi _0)$ and
$\phi_2\colon\rr{d}\to\rr{d}\colon \xi \mapsto  \varphi^\prime_x(x_0,\xi)$ 
are $\SG$ maps (with $\SG^0$ parameter dependence), from $\rr{d}$ to itself. 
If $\varphi\in\Phr$, $\phi_1$ and $\phi_2$ give rise to $\SG$ diffeomorphism 
with $\SG^{0}$ parameter dependence.
\end{prop}

\par

For any $\varphi \in \Ph$, the operators $\Theta _{1,\fy}$ and $\Theta _{2,\fy}$
are defined by
$$
(\Theta _{1,\fy}f)(x,\xi ) \equiv f(\fy '_\xi (x,\xi ), \xi)
\quad \text{and}\quad
(\Theta _{2,\fy}f)(x,\xi ) \equiv f(x, \fy '_x(x,\xi )),
$$
when $f\in C^1(\rr {2d})$, and remark that
the modified weights
\begin{equation}\label{omegaVarphiDef}
(\Theta _{1,\fy}\omega )  (x,\xi ) = \omega (\fy '_\xi (x,\xi ), \xi)
\quad \text{and}\quad
(\Theta _{2,\fy}\omega ) (x,\xi ) = \omega (x, \fy '_x(x,\xi )), 
\end{equation}
will appear frequently in the sequel. In the following lemma we
show that these weights belong to the same classes of weights
as $\omega$, provided they additionally fulfill
\begin{equation}\label{WeightPhaseCond}
\Theta _{1,\fy}\omega  \asymp \Theta _{2,\fy}\omega 
\end{equation}
when $\fy$ is the involved phase function. That is,
\eqref{WeightPhaseCond} is a sufficient condition to obtain
$(\phi_1,1)$- and/or $(\phi_2,2)$-invariance of $\omega$
in the sense of Definition \ref{def:omegainv}, depending on the values of
the parameters $r,\rho\ge0$.

\par

\begin{lemma}
Let $\fy$ be a simple phase on $\rr {2d}$, $r,\rho \in [0,1]$ be such that
$r=1$ or $\rho =1$, and let $\Theta _{j,\fy}\omega$, $j=1,2$, be as in
\eqref{omegaVarphiDef}, where $\omega \in \mascP
_{r,\rho}(\rr {2d})$ satisfies \eqref{WeightPhaseCond}.
Then
\begin{equation*}
	\Theta _{j,\fy}\omega \in \mascP _{r,\rho}(\rr {2d}),\quad j=1,2.
\end{equation*}
\end{lemma}

\par

Here and in what follows we let
$$
^{t}a(x,\xi)=a(\xi,x)\quad  \text{and} \quad (a^*)(x,\xi)
=\overline{a(\xi,x)},
$$
when $a(x,\xi)$ is a function.

\par

\subsection{Generalised Fourier integral operators of $\SG$
type}\label{subs:2.2}
In analogy with the definition of generalized $\SG$
pseudo-differential operators, recalled in Subsection
\ref{subs:1.1}, we define the class of Fourier integral
operators we are interested in terms of their distributional
kernels. These belong to a class of tempered
oscillatory integrals, studied in \cite{CoSch}. Thereafter we
prove that they posses convenient mapping properties.

\par

\begin{defn}\label{def:sgfios}
Let $\omega \in \mascP_{r,\rho}(\rr{2d})$ satisfy
\eqref{WeightPhaseCond}, $r,\rho\ge0$, $r+\rho>0$,
$\varphi\in\Ph$, $a,b\in\SG^{(\omega)}_{r,\rho}(\rr{2d})$.
\begin{enumerate}
\item The generalized Fourier integral operator $A=\op _\fy (a)$ of
\emph{$\SG$ type I} (\emph{$\SG$ FIOs of type I}) with phase $\varphi$
and amplitude $a$ is the linear continuous operator from $\cS(\rr d)$ to
$\cS^\prime(\rr d)$ with distribution kernel $K_A\in
\cS^\prime(\rr {2d})$ given by
\[
	K_A(x,y)=(2\pi)^{-d/2}(\cF_2(e^{i\fy}a))(x,y)\text ;
\]

\vrum

\item The generalized Fourier integral operator $B=\op _\fy ^*(b)$ of
\emph{$\SG$ type II} (\emph{$\SG$ FIOs of type II}) with phase
$\varphi$ and amplitude $b$ is the linear continuous operator
from $\cS(\rr d)$ to $\cS^\prime(\rr d)$ with distribution kernel
$K_B\in \cS^ \prime(\rr {2d})$ given by
\[
	K_B(x,y)=(2\pi)^{-d/2}(\cF^{-1}_2(e^{-i\fy }\overline b))(y,x).
\]
\end{enumerate}
\end{defn}

\par

Evidently, if $f \in \cS(\rr d)$, and $A$ and $B$ are the operators in Definition
\ref{def:sgfios}, then
\begin{align}
Af(x) &= \op_\varphi(a)u(x) = (2\pi)^{-d/2}\int e^{i \varphi(x, \xi)}
\, a(x, \xi) \, ({\cF{f}})(\xi)\,d \xi,\label{eq:0.1}
\intertext{and}
Bf(x) &= \op^*_\varphi(b)u(x)\notag
\\[1ex]
 &=(2\pi)^{-d} \iint e^{i(\langle x, \xi)- \varphi(y, \xi))}
                            \, \overline{b(y, \xi)} \, f(y) \,dy d\xi.\label{eq:0.1.0}
\end{align}

\par

\begin{rem}\label{rem:sgsymm}
In the sequel the formal ($L^2$-)adjoint of an operator $Q$ is denoted by $Q^*$.
By straightforward computations it follows that the $\SG$ type I
and $\SG$ type II operators are formal adjoints to each others, provided the
amplitudes and phase functions are the same. That is, if $b$ and $\varphi$
are the same as in Definition \ref{def:sgfios}, then
$\op^*_\varphi(b)=\op_\varphi(b)^*$.

\par

Obviously, for any $\omega \in \mascP_{r,\rho}(\rr {2d})$,
$^{t}\omega=\omega^*$ is also an admissible weight which belongs to
$\mascP_{\rho,r}(\rr {2d})$. Similarly, for arbitrary
$\varphi \in \Ph$ and $a \in \SG^{(\omega)}_{r,\rho}(\rr {2d})$, we have
$^{t}\varphi=\varphi ^*\in\Ph$ and $^{t}a, a^*\in\SG ^{(\omega ^*)} _{\rho,r}(\rr {2d})$.
Furthermore, by Definition \ref{def:sgfios} we get
\begin{equation}\label{eq:typeI-II}
\begin{gathered}
\op_\varphi^*(b) =  \cF^{-1} \circ \op_{- \varphi ^*}(b^*) \circ \cF^{-1}
\\[1ex]
\Longleftrightarrow
\\[1ex]
\op_{\varphi}(a) =  \cF \circ \op _{- \varphi ^*}^*(a^*) \circ \cF.
\end{gathered}
\end{equation}
\end{rem}

\par

The following result shows that type I and type II operators
are linear and continuous from $\cS (\rr d)$ to itself, and
extendable to linear and continuous operators from
$\cS^\prime(\rr d)$ to itself.

\par

\begin{thm}
\label{thm:2.1}
Let $a$, $b$ and $\varphi$ be the same as in Definition \ref{def:sgfios}.
Then $\op _\fy (a)$ and $\op _\fy ^*(b)$ are linear and continuous
operators on $\cS (\rr d)$, and uniquely extendable to linear and
continuous operators on $\cS '(\rr d)$.
\end{thm}

\subsection{Composition with pseudo-differential operators of $\SG\!$ type.}
\label{subs:2.3}
The composition theorems presented in this and the subsequent subsections
are variants of those originally appeared in \cite{coriasco}. 
The notation used in the statements of the composition theorems are those
introduced in Subsections \ref{subs:1.3}, \ref{subs:2.1} and \ref{subs:2.2}.
The proofs and more details can be found in \cite{CoTo2}.

\begin{thm}
\label{thm:0.1}
Let $r_j,\rho _j\in [0,1]$, $\varphi \in \Ph$ and let $\omega _j
\in \mascP_{r_j,\rho _j}(\rr {2d})$, $j=0,1,2$, be such
that
$$
\rho _2=1,
\quad r_0=\min\{r_1,r_2,1\} ,\quad \rho _0=\min\{ \rho_1,1\},
\quad \omega _0 =\omega_1\cdot (\Theta _{2,\fy}\omega _2),
$$
and $\omega _2\in \mathscr{P}_{r,1}(\rr{2d})$ is
$(\phi,2)$-invariant with respect to $\phi\colon
\xi\mapsto\varphi^\prime_x(x,\xi)$.
Also let $a \in \SG^{(\omega _1)} _{r_1,\rho_1}(\rr {2d})$,
$p \in \SG^{(\omega _2)}_{r_2,1}(\rr {2d})$, and let
\begin{equation}
\label{eq:0.4}
\psi(x,y,\xi) = \varphi(y,\xi) - \varphi(x,\xi) -
\scal{ y - x}{\varphi '_x(x,\xi)}.
\end{equation}
Then
\begin{alignat*}{2}
\op(p) \circ \op_\varphi(a) &= \op_{\varphi}(c) \operatorname{Mod}
\op _\varphi (\SG ^{(\omega \vartheta _{0,-\infty })}_0 ),& \quad r_1=0,
\\[1ex]
\op(p) \circ \op_\varphi(a) &= \op_{\varphi}(c) \operatorname{Mod}
\op (\mathscr S ),& \quad r_1>0,
\end{alignat*}
where $c \in \SG^{(\omega _0)}_{r_0,\rho _0}(\rr {2d})$
admits the asymptotic expansion
\begin{equation}
\label{eq:0.3}
c(x,\xi) \sim \sum_{\alpha} \frac{i^{|\alpha|}}{\alpha!} 
(D^\alpha_\xi p)(x, \varphi '_x(x,\xi))
\,D^\alpha_y \!\!\left[ e^{i \psi(x,y,\xi)} a(y,\xi) \right]_{y=x}.
\end{equation}
\end{thm}

\begin{thm}
\label{thm:3.1}
Let $r_j,\rho _j\in [0,1]$, $\varphi \in \Ph$ and let $\omega _j
\in \mascP_{r_j,\rho _j}(\rr {2d})$, $j=0,1,2$, be such
that
$$
r_2=1,
\quad r_0=\min\{r_1,1\} ,\quad \rho _0=\min\{ \rho_1,\rho _2,1\},
\quad \omega _0 =\omega_1\cdot (\Theta _{1,\fy} \omega _2),
$$
and $\omega _2\in \mathscr{P}_{r,1}(\rr{2d})$ is
$(\phi,1)$-invariant with respect to $\phi \colon
x \mapsto \varphi^\prime_\xi (x,\xi)$.
Also let $a \in \SG^{(\omega _1)} _{r_1,\rho _1}(\rr {2d})$ and
$p \in \SG^{(\omega _2)}_{1,\rho _2}(\rr {2d})$.
Then
\begin{alignat*}{2}
\op _\varphi(a) \circ \op (p) &= \op_{\varphi}(c) \operatorname{Mod}
\op _\varphi (\SG ^{(\omega \vartheta _{-\infty ,0})}_0 ),& \quad \rho _1=0,
\\[1ex]
\op_\varphi(a) \circ \op (p) &= \op_{\varphi}(c) \operatorname{Mod}
\op (\mathscr S ),& \quad \rho _1>0,
\end{alignat*}
where the transpose ${^t}c$ of $c \in \SG^{(\omega _0)}
_{r_0,\rho _0}(\rr {2d})$ admits the asymptotic expansion
\eqref{eq:0.3}, after $p$ and $a$ have been replaced by
${^t}p$ and ${^t}a$, respectively.
\end{thm}

\par

\begin{thm}
\label{thm:3.2}
Let $r_j,\rho _j\in [0,1]$, $\varphi \in \Ph$ and let $\omega _j
\in \mascP_{r_j,\rho _j}(\rr {2d})$, $j=0,1,2$, be such
that
$$
\rho _2=1,
\quad r_0=\min\{r_1,r_2,1\} ,\quad \rho _0=\min\{ \rho_1,1\},
\quad \omega _0 =\omega_1\cdot (\Theta _{2,\fy} \omega _2),
$$
and $\omega _2\in \mathscr{P}_{r,1}(\rr{2d})$ is
$(\phi,2)$-invariant with respect to $\phi \colon
\xi \mapsto \varphi ^\prime _x(x,\xi)$.
Also let $b \in \SG^{(\omega _1)} _{r_1,\rho_1}(\rr {2d})$,
$p \in \SG^{(\omega _2)}_{r_2,1}(\rr {2d})$, $\psi$
be the same as in \eqref{eq:0.4}, and let $q \in \SG^{(\omega _2)}
_{r_2,1}(\rr {2d})$ be such that
\begin{equation}
\label{eq:1.11.2}
q(x,\xi)\sim\sum_{\alpha}\frac{i^{|\alpha|}}{\alpha!}D^\alpha_x D^\alpha_\xi\overline{p(x,\xi)}.
\end{equation}

Then
\begin{alignat*}{2}
\op _\varphi ^*(b) \circ \op(p) &= \op_{\varphi}(c) \operatorname{Mod}
\op _\varphi ^*(\SG ^{(\omega \vartheta _{0,-\infty })}_0 ),& \quad r_1=0,
\\[1ex]
\op _\varphi ^*(b) \circ \op(p) &= \op_{\varphi}(c) \operatorname{Mod}
\op (\mathscr S ),& \quad r_1>0,
\end{alignat*}
where $c \in \SG^{(\omega _0)}_{r_0,\rho _0}(\rr {2d})$
admits the asymptotic expansion
\begin{equation}
\label{eq:3.38}
c(x,\xi) \sim \sum_{\alpha} \frac{i^{|\alpha|}}{\alpha!} 
(D^\alpha_\xi q)(x, \varphi^\prime_x(x,\xi))
D^\alpha_y \!\!\left[ e^{i \psi(x,y,\xi)} b(y,\xi) \right ]_{y=x}.
\end{equation}
\end{thm}

\par

\begin{thm}
\label{thm:3.3}
Let $r_j,\rho _j\in [0,1]$, $\varphi \in \Ph$ and let $\omega _j
\in \mascP_{r_j,\rho _j}(\rr {2d})$, $j=0,1,2$, be such
that
$$
r_2=1,
\quad r_0=\min\{r_1,1\} ,\quad \rho _0=\min\{ \rho_1,\rho _2,1\},
\quad \omega _0 =\omega _1\cdot (\Theta _{1,\fy} \omega _2),
$$
and $\omega _2\in \mathscr{P}_{r,1}(\rr{2d})$ is
$(\phi,1)$-invariant with respect to $\phi \colon
x \mapsto \varphi^\prime_\xi (x,\xi)$.
Also let $a \in \SG^{(\omega _1)} _{r_1,\rho _1}(\rr {2d})$ and
$p \in \SG^{(\omega _2)}_{1,\rho _2}(\rr {2d})$.
Then
\begin{alignat*}{2}
\op (p)\circ \op _\varphi ^*(b) &= \op_{\varphi}(c) \operatorname{Mod}
\op _\varphi ^*(\SG ^{(\omega \vartheta _{-\infty ,0})}_0 ),& \quad \rho _1=0,
\\[1ex]
\op (p)\circ \op _\varphi ^*(b) &= \op_{\varphi}(c) \operatorname{Mod}
\op (\mathscr S ),& \quad \rho _1>0,
\end{alignat*}
where the transpose ${^t}c$ of $c \in \SG^{(\omega _0)}
_{r_0,\rho _0}(\rr {2d})$ admits the asymptotic expansion
\eqref{eq:3.38}, after $q$ and $b$ have been replaced by
${^t}q$ and ${^t}b$, respectively.
\end{thm}

\par

\subsection{Composition between $\SG$ FIOs of type I and type II}
\label{subs:2.4}
The subsequent Theorems \ref{thm:3.4} and \ref{thm:3.5} deal with the
composition of a type I operator with a type II operator, and show that such compositions
are pseudo-differential operators with symbols in natural classes. 

\par

The main difference, with respect to the arguments in \cite{coriasco}
for the analogous composition results, is that we again make use, in both cases, of 
the generalized asymptotic expansions introduced in Definition \ref{def:gensgasexp}.
This allows to overcome the additional difficulty, not arising there, that the amplitudes
appearing in the computations below involve weights which are still polynomially bounded, but
which do not satisfy, in general, the moderateness condition \eqref{moderate}. On the 
other hand, all the terms appearing in the associated asymptotic expansions belong 
to $\SG$ classes with weights of the form $\widetilde{\omega}_{2,\varphi}\cdot\vartheta_{-k,-k}$,
where $\widetilde{\omega}=\omega_1\cdot\omega_2$, which can be handled through the results 
in \cite{CoTo}.

\par

Let $S_\fy$, $\fy \in \Ph$, be the operator defined by the formulae
\begin{equation}\label{SvarphiDef}
\begin{gathered}
(S_\fy f)(x,y,\xi ) = f(x,y,\Phi (x,y,\xi ))\cdot \left | \det \Phi '_\xi (x,y,\xi )\right |
\\[1ex]
\text{where}\quad
\int _0^1 \fy _x'(y+t(x-y),\Phi (x,y,\xi ))\, dt =\xi .
\end{gathered}
\end{equation}
That is, for every fixed $x,y\in \rr d$, $\xi \mapsto \Phi (x,y,\xi )$ is the inverse of the map
\begin{equation}\label{eq:diffeo}
\xi \mapsto \int _0^1 \fy _x'(y+t(x-y),\xi )\, dt .
\end{equation}
Notice that, as proved in \cite{coriasco}, the map \eqref{eq:diffeo} is indeed invertible for $(x,y)$ belonging to a suitable
neighborhood of the diagonal $y=x$ of $\rr{d}\times\rr{d}$, and it turns out to be a $\SG$ diffeomorphism with $\SG^0$
parameter dependence. We also recall, from \cite{coriasco}, the definition of the 
$\SG$ compatible cut-off functions localizing to such neighborhoods.

\par

\begin{defn}
\label{def:1.2.2}
The sets $\Xi^\Delta(k)$, $k > 0$, of the $\SG$ compatible cut-off functions
along the diagonal of $\rr{d}\times\rr{d}$, consist of all 
$\chi = \chi(x,y) \in  \SG ^{0,0}_{1,1}(\rr {2d})$ such that
\begin{equation}
\label{eq:1.1.3}
\begin{array}{rcl}
        |y-x| \le k\norm{x}/2 & \Longrightarrow & \chi(x,y) = 1,
        \\[1ex]
        |y-x| >         k      \norm{x} & \Longrightarrow & \chi(x,y) = 0.
\end{array}
\end{equation}
If not otherwise stated, we always assume $k \in (0,1)$.
\end{defn}

\par

\begin{thm}
\label{thm:3.4}
Let $r_j\in [0,1]$, $\varphi \in \Ph$ and let $\omega _j
\in \mascP_{r_j,1}(\rr {2d})$, $j=0,1,2$, be such that $\omega _1$ and
$\omega _2$ are $(\phi,2)$-invariant
with respect to $\phi \colon
\xi \mapsto (\varphi ^\prime _x)^{-1}(x,\xi)$,
$$
\quad r_0=\min\{r_1,r_2,1\} \quad \text{and}
\quad \omega _0(x,\xi ) =\omega_1(x,\phi (x,\xi ))\omega_{2}(x,\phi (x,\xi )),
$$
Also let $a \in \SG^{(\omega _1)} _{r_1,1}(\rr {2d})$ and
$b \in \SG^{(\omega _2)}_{r_2,1}(\rr {2d})$.
Then
\begin{equation*}
\op _\varphi (a) \circ \op _\varphi ^*(b) = \op (c),
\end{equation*}
for some $c \in \SG^{(\omega _0)}_{r_0,1}(\rr {2d})$.
Furthermore, if $\ep \in (0,1)$, $\chi \in \Xi ^\Delta (\ep )$,
$c_0(x,y,\xi )= a(x,\xi )b(y,\xi )\chi (x,y)$ and $S_\fy$ is given by
\eqref{SvarphiDef}, then $h$ admits the asymptotic expansion
\begin{equation*}
c(x,\xi) \sim \sum _{\alpha} \frac{i^{|\alpha |}}{\alpha !} 
(D^\alpha _yD^\alpha_\xi (S_\fy c))(x,y,\xi) \big |_{y=x}.
\end{equation*}
\end{thm}

To formulate the next result we modify the operator $S_\fy$ in \eqref{SvarphiDef}
such that it fulfills the formulae
\begin{equation}\label{SvarphiDef2}
\begin{gathered}
(S_\fy f)(x,\xi ,\eta ) = f(\Phi (x,y,\xi ),\xi,\eta )\cdot \left | \det \Phi '_x (x,\xi ,\eta )\right |
\\[1ex]
\text{where}\quad
\int _0^1 \fy _\xi '(\Phi (x,\xi ,\eta ), \eta +t(\xi -\eta ))\, dt =x .
\end{gathered}
\end{equation}

\par

\begin{thm}
\label{thm:3.5}
Let $\rho _j\in [0,1]$, $\varphi \in \Phr$ and let $\omega _j
\in \mascP_{1,\rho _j}(\rr {2d})$, $j=0,1,2$, be such that $\omega _1$ and
$\omega _2$ are $(\phi,1)$-invariant
with respect to $\phi \colon
x \mapsto (\varphi ^\prime _\xi )^{-1}(x,\xi)$,
$$
\quad \rho _0=\min\{ \rho _1,\rho _2,1\} \quad \text{and}
\quad \omega _0(x,\xi ) =\omega_1(\phi (x,\xi ),\xi )\omega_{2}(\phi (x,\xi ),\xi ),
$$
Also let $a \in \SG^{(\omega _1)} _{1,\rho _1}(\rr {2d})$ and
$b \in \SG^{(\omega _2)}_{1,\rho _2}(\rr {2d})$.
Then
\begin{equation*}
\op _\varphi ^*(b) \circ \op _\varphi (a) = \op (c),
\end{equation*}
for some $c \in \SG^{(\omega _0)}_{1,\rho _0}(\rr {2d})$.
Furthermore, if $\ep \in (0,1)$, $\chi \in \Xi ^\Delta (\ep )$,
$c_0(x,\xi ,\eta )= a(x,\xi )b(x,\eta )\chi (\xi ,\eta )$ and $S_\fy$ is given by
\eqref{SvarphiDef2}, then $h$ admits the asymptotic expansion
\begin{equation*}
c(x,\xi) \sim \sum _{\alpha} \frac{i^{|\alpha |}}{\alpha !} 
(D^\alpha _xD^\alpha_\eta (S_\fy c))(x,\xi ,\eta ) \big |_{\eta =\xi}.
\end{equation*}
\end{thm}
\subsection{Elliptic FIOs of generalized $\SG$ type and parametrices. Egorov Theorem}
\label{sec:2.5}
The results about the parametrices of the subclass of
generalized ($\SG$) elliptic  Fourier integral operators
are achieved in the usual way, by means of the composition
theorems in Subsections \ref{subs:2.3} and \ref{subs:2.4}.
The same holds for the versions of the Egorov's theorem
adapted to the present situation. The additional conditions,
compared with the statements in \cite{coriasco}, concern
the invariance of the weights, so that the hypotheses of
the composition theorems above are fulfilled. 

\begin{defn}
\label{def:3.1}
A type I  or a type II $\SG$ FIO, $\op_{\varphi}(a)$ or $\op^*_{\varphi}(b)$, respectively,
is said ($\SG$) elliptic if $\varphi \in \Phr$ and the amplitude $a$, respectively $b$, is ($\SG$) elliptic. 
\end{defn}
\begin{lemma}
\label{lemma:3.15}
Let a type I $\SG$ FIO $\op_{\varphi}(a)$ be elliptic, with 
$a\in\SG^{(\omega)}_{1,1}(\rr{2d})$. Assume that $\omega$ is $\phi$-invariant, 
$\phi=(\phi_1,\phi_2)$, where $\phi_2$ and $\phi_1$ are the $\SG$ diffeomorphisms
appearing in Theorems \ref{thm:3.4} and \ref{thm:3.5}, respectively.
Then, the two pseudo-differential operators $\op_{\varphi}(a)\circ\op^*_{\varphi}(a)$ and $\op
^*_{\varphi}(a)\circ\op_{\varphi}(a)$ are $\SG$ elliptic.
\end{lemma}

\par

\begin{thm}
\label{thm:3.8}
Let $\varphi \in \Phr$, $a\in \SG ^{(\omega)}_{1,1}(\rr{2d})$, with $a$ $\SG$ elliptic.
Assume that $\omega$ is $\phi$ invariant, 
$\phi=(\phi_1,\phi_2)$, where $\phi_2$ and $\phi_1$ are the $\SG$ diffeomorphisms
appearing in the Theorems \ref{thm:3.4} and \ref{thm:3.5}, respectively.
Then, the elliptic $\SG$ FIOs $\op_{\varphi}(a)$ and $\op^*_{\varphi}(a)$ admit a parametrix.
These are elliptic $\SG$ FIOs of type II and type I, respectively.
\end{thm}

\noindent
As usual, in the next two results we need
the canonical transformation $\phi\colon(x,\xi)\mapsto(y,\eta)$ generated by
the phase function $\varphi$, namely
\begin{equation}\label{eq:phi}
	\begin{cases}
		\xi=\varphi^\prime_x(x,\eta)
		\\
		y=\varphi^\prime_\xi(x,\eta).
	\end{cases}
\end{equation}

\begin{thm}
\label{thm:3.21}
Let $A = \op_{\varphi}(a)$ be an $\SG$ FIO of type I with
$a \in \SG^{(\omega_0)}_{1,1}(\rr{2d})$ and $P = \Op{p}$  a pseudo-differential operator with
$p \in \SG^{(\omega)}_{1,1}(\rr{2d})$. Assume that $\omega$ is $\phi$-invariant,
where $\phi$ is the canonical transformation \eqref{eq:phi}, associated with $\varphi$. Assume
also that $\omega_0$ is $(\tilde{\phi},2)$-invariant, where $\tilde{\phi}\colon\xi\mapsto(\varphi '_{x})^{-1}(x,\xi)$.
Then, setting $\eta = (\varphi '_{x})^{-1}(x,\xi)$ we have
\begin{equation}
\label{eq:3.72.3bis}
\begin{aligned}
\Sym{A\circ P\circ A^{*}}(x,\xi) &= p( \varphi^\prime_{\xi}(x, \eta), \eta)\,|a(x,\eta)|^2\,
|\det \varphi^{\prime\prime}_{x\xi}(x,\eta)|^{-1}
\\
& \mod \SG^{(\widetilde{\omega}\cdot\vartheta_{-1,-1})}_{1,1}(\rr{2d}),
\end{aligned}
\end{equation}
which is an element of $\SG^{(\widetilde{\omega})}_{1,1}(\rr{2d})$ with 
\[
	\widetilde{\omega}(x,\xi)=\omega(\phi(x,\xi))\cdot
	\omega_0(x,(\varphi '_{x})^{-1}(x,\xi))^2.
\]
\end{thm}
\begin{thm}
\label{thm:3.21ell}
Let $A = \op_{\varphi}(a)$ be an elliptic $\SG$ FIO of type I with
$a \in \SG^{(\omega_0)}_{1,1}(\rr{2d})$ and $P = \Op{p}$  a pseudo-differential operator with
$p \in \SG^{(\omega)}_{1,1}(\rr{2d})$. Assume that $\omega$ is $\phi$-invariant,
where $\phi$ is the canonical transformation \eqref{eq:phi}, associated with 
$\varphi$. Then,  we have
\begin{equation}
\label{eq:3.72.3}
\Sym{A\circ P\circ A^{-1}}(x,\xi) = p( \phi(x,\xi)) \!\!\mod 
\SG^{(\widetilde{\omega}\cdot\vartheta_{-1,-1})}_{1,1}(\rr{d}),
\end{equation}
with $\widetilde{\omega}(x,\xi)=\omega(\phi(x,\xi))$.
\end{thm}
%

\section{Continuity on Lebesgue and modulation spaces}\label{sec3}
%

In this section we recall some basic facts about continuity properties for
Fourier integral operators when acting on Lebesgue and modulation spaces.
We also use the analysis in previous sections in combination with certain
lifting properties for modulation spaces in order to establish weighted versions
of continuity results for Fourier integral operators on modulation spaces.

\par

\subsection{Continuity on Lebesgue spaces}

\par

We start by considering the following result, which, for trivial Sobolev
parameters, is related to Theorem 2.6 in \cite{CoRu}. A direct proof of the
$L^2(\rr{d})\to L^2(\rr{d})$ boundedness of $\op_\varphi(a)$
for $a\in\SG^{0,0}_{1,1}(\rr{d})$ and a regular phase function $\varphi\in\Phr$
was given in \cite{coriasco}. A similar argument actually holds 
for $a\in\SG^{0,0}_{r,\rho}(\rr{d})$, $r,\rho\ge0$, and is given in
\cite{CoTo2} (see also \cite{RuSu}). Here
$B_r(x_0)$ is the open ball with center at $x_0\in \rr d$ and radius $r$.

\par

\begin{thm}\label{FIOSobcont}
Let $\sigma _1,\sigma _2\in \mathbf R$, $p\in (1,\infty )$ and $m,\mu \in
\mathbf R$ be such that
$$
m \le -(d-1)\left | \frac 1p-\frac 12 \right | ,\quad \mu \le -(d-1)\left | \frac 1p-\frac 12
\right | +\sigma _1-\sigma _2.
$$
Also let $\fy \in \SG ^{1,1}_{1,1} (\rr {2d})$ be such that for some constants
$c>0$ and $R>0$ and every multi-index $\alpha$ it holds
\begin{alignat*}{2}
|\det \varphi ''_{x , \xi} (x,\xi )| &\ge c, & \qquad
|\partial _x^\alpha \fy (x, \xi )| & \lesssim \eabs x^{1-|\alpha |}\eabs \xi
\\[1ex]
\eabs {\fy '_x (x,\xi )} &\asymp \eabs \xi , & \qquad \eabs {\fy '_\xi (x,\xi )}
&\asymp \eabs x,
\intertext{and}
\fy (x,t\xi ) &= t\fy (x,\xi ),&\quad  x,\xi \in \rr d,\ |\xi | &\ge R,\ t\ge 1.
\end{alignat*}
If $a\in \SG ^{m,\mu}_{1,1}(\rr {2d})$ is supported outside $\rr d\times B_r(0)$
for some $r>0$, then $\op_\varphi(a)$ extends to a
continuous operator from $H^p_{\sigma _1}(\rr d)$ to $H^p_{\sigma _2}(\rr d)$.
\end{thm}

\par

\begin{proof}
Let $T = \eabs D^{\sigma _2}\circ \op_\varphi(a) \circ \eabs D^{-\sigma _1}$. Since
$$
\eabs D^{\sigma _2}\, : \, H^p_{\sigma _2}\to L^p\quad \text{and}\quad
\eabs D^{-\sigma _1}\, : \, L^p\to H^p_{\sigma _1}
$$
are continuous bijections, the result follows if we prove that $T$ is continuous
on $L^p$.

\par

By Theorems \ref{thm:3.1} and \ref{thm:3.2} it follows that
$$
T=\op_\varphi(a_1)\mod \op (\sS ),
$$
where $a_1\in \SG ^{m,\mu_0}_{1,1}(\rr {2d})$ with
$$
\mu _0\le -(d-1)\left | \frac 1p-\frac 12 \right |.
$$
Furthermore, by the symbolic calculus and the fact that $a$ is supported
outside $\rr d\times B_r(0)$ we get
$$
\op_\varphi(a_1) = \op_\varphi(a_2)\mod \op (\sS ),
$$
where $a_2\in \SG ^{m,\mu_0}_{1,1}(\rr {2d})$ is supported
outside $\rr d\times B_r(0)$. Hence
$$
T=\op_\varphi(a_2)+\op (c),
$$
where $c\in \sS$, giving that $\op (c)$ is continuous on $L^p$.

\par

Since $\op_\varphi(a_2)$ is continuous on $L^p$, by
\cite[Theorem 2.6]{CoRu} and its proof, the result follows.
\end{proof}

\par

\subsection{Continuity on modulation spaces}

\par

Next we consider continuity properties on modulation spaces. The following result
extends Theorem 1.2 in \cite{CorNicRod1}. Here we let $M^\infty _{0,(\omega )}
(\rr d)$ be the completion of $\mathscr S(\rr d)$ under the norm $\nm
\cdo {M^\infty _{(\omega )}}$. We also say that a (complex valued) Gauss function
$\Psi$ is non-degenerate, if $|\Psi |$ tends to zero at infinity. 

\par

\begin{thm}\label{FIOmodcont}
Let $m,\mu \in \mathbf R$ and $1\le p<\infty$ be such that
$$
m\le -d\left | \frac 12-\frac 1p  \right | ,\quad \mu \le -d\left | \frac 12-\frac 1p 
\right | ,
$$
and let $\omega _j\in \mascP _{1,1}(\rr {2d})$, $j=0,1,2$, be such that
$$
\omega _0(x,\xi ) \lesssim \frac {\omega _1(\varphi^\prime_\xi(x,\xi),\xi )}{\omega
_2(x,\varphi^\prime_x(x,\xi) )}\eabs x^m \eabs \xi ^\mu.
$$
Also let $a\in \SG ^{(\omega _0)}_{1,1}(\rr {2d})$ and $\fy\in\Phr$, and assume
that $\omega_j$ is $(\phi_j,j)$-invariant, $j=1,2$, with
$\phi_1\colon x\mapsto\varphi^\prime_\xi(x,\xi)$ 
and $\phi_2\colon\xi\mapsto\varphi^\prime_x(x,\xi)$.
Then $\op_\varphi(a)$ is uniquely extendable to a continuous
map from $M^p_{(\omega _1)}(\rr d)$ to $M^p_{(\omega _2)}(\rr d)$ and from
$M^\infty _{0,(\omega _1)}(\rr d)$ to $M^\infty _{0,(\omega _2)}(\rr d)$.
\end{thm}

\par

\begin{proof}
Let $\Psi$ be a Gaussian, and let $T_1$ and $T_2$ be the operators, defined
by the formulas
$$
(T_1f,g) = (\omega _1^{-1}V_\Psi f,V_\Psi g)\quad \text{and}\quad(T_2f,g) =
(\omega _2V_\Psi f,V_\Psi g).
$$
Then it follows from Theorem 1.1 in \cite{GT} that $T_1$ and $T_2$
on $\mathscr S$ are uniquely extendable to continuous bijections between
$M^{p}$ to $M^{p}_{(\omega _1)}$, and from $M^{p}_{(\omega _2)}$
to $M^{p}$. Since $\cS$ is dense in $M^p_{(\omega _j)}$ and in
$M^\infty _{(\omega _j)}$, the result follows if we prove
$$
\nm {(T_2 \circ \op_\varphi(a)\circ T_1)f}{M^p}\lesssim \nm f{M^p},\qquad f\in \cS .
$$

\par

For some non-degenerate Gauss function $\Phi$ which depends on $\Psi$
we have
$$
T_j = \op ( a_j ),\ j=1,2,\quad \text{where}\quad
a_1= ((\omega _1)^{-1})*\Phi \quad \text{and}\quad
a_2 = \omega _2*\Phi .
$$
Furthermore, using the fact that $\omega _j\in \mascP _{1,1}$, it follows by
straight-forward computations that $a_1\in \SG ^{(1/\omega _1)}_{1,1}$ and
$a_2\in \SG ^{(\omega _2)}_{1,1}$.

\par

By using these facts in combination with Theorems \ref{thm:3.1} and
\ref{thm:3.2}, we get
$$
T_2 \op_\varphi(a)\circ T_1 =T_2\circ (\op_\varphi(h_1) +S_1) = \op_\varphi(h_2)+
S_2+T_2\circ S_1,
$$
for some operators $S_j\in \op (\cS )$, $j=1,2$, where
$$
h_1\in \SG ^{(\omega _0/\widetilde{\omega} _1)}_{1,1}\quad \text{and}\quad
h_2\in \SG ^{(\omega _0\widetilde{\omega} _2/\widetilde{\omega} _1)}_{1,1}\subseteq
\SG ^{m,\mu}_{1,1},
$$
$\widetilde{\omega}_1(x,\xi)=\omega_1(\varphi^\prime_\xi(x,\xi),\xi)$, 
$\widetilde{\omega}_2(x,\xi)=\omega_2(x,\varphi^\prime_x(x,\xi) )$. Since
$$
T_2\circ S_1 \in \op (\SG ^{(\omega _2)}_{1,1})\circ \op (\mathscr S)
\subseteq \op (\mathscr S),
$$
it follows that
$$
T_2 \circ \op_\varphi(a)\circ T_1 = \op_\varphi(h_2)+S_0,
$$
where $S_0\in \op (\cS )$, giving that $S_0$ is continuous on $M^p$.
Furthermore, the fact that $h_2\in \SG ^{m,\mu}_{1,1}$ and
Theorem 1.2 in \cite{CorNicRod1} imply that
$$
\nm {\op_\varphi(h_1)f}{M^{p}}\lesssim \nm f{M^{p}},\quad f\in \cS .
$$
This gives the result.
\end{proof}

\par

The continuity properties of $\SG$ pseudo-differential operators on
modulation spaces, as well as the propagation of the global
wave-front sets under their action, shortly recalled in the next
Section \ref{sec4}, motivate the next definition, originally given in \cite{CJT2}. 

\par

\begin{defn}\label{def:admspace}
Let $r,\rho \in [0,1]$, $t\in \mathbf R$, $\cB$ be a topological vector space
of distributions on $\rr d$ such that
$$
\mathscr S(\rr d)\subseteq \cB  \subseteq \mathscr S'(\rr d)
$$
with continuous embeddings. Then $\cB$ is called \emph{$\SG$-admissible
(with respect to $r$, $\rho$ and $d$)} when $\op _t(a)$ maps $\cB$ continuously
into itself, for every $a\in \SG ^{0,0}_{r,\rho}(\rr{d})$.
If $\cB$ and $\cC$ are $\SG$-admissible with respect to $r$, $\rho$ and $d$,
and $\omega _0\in \mathscr P_{r,\rho}(\rr {2d})$, then the pair $(\cB ,\cC )$ is
called \emph{$\SG$-ordered (with respect to $\omega _0$)}, when the mappings
$$
\op _t(a)\, :\, \cB \to \cC \quad \text{and}\quad \op _t(b)\, :\, \cC \to \cB
$$
are continuous for every $a\in \SG ^{(\omega _0)}_{r,\rho}(\rr {2d})$ and
$b\in \SG ^{(1/\omega _0)}_{r,\rho}(\rr {2d})$.
\end{defn}

The following definition, which extends Definition \ref{def:admspace} to the case of
generalized Fourier integral operators, is justified by Theorems \ref{FIOSobcont}
and \ref{FIOmodcont}.

\par

\begin{defn}\label{admspacesdef}
Let $\fy \in \SG ^{1,1}_{1,1}(\rr {2d})$ be a regular phase function,
and $\cB$, $\cB_1$, $\cB_2$, $\cC$, $\cC_1$, $\cC_2$, be
$\SG$-admissible with respect to $r$, $\rho$ and $d$. Let also $\omega _0,\omega_1,\omega_2\in
\mathscr P_{r,\rho}(\rr {2d})$, and $\Omega \subseteq \rr d$ be open. Then
the pair $(\cB ,\cC )$ is called \emph{weakly-I $\SG$-ordered (with respect to
$(r,\rho ,\omega _0,\varphi ,\Omega )$)}, when the mapping
$$
\op_\varphi(a)\, :\, \cB \to \cC
$$
is continuous for every $a\in \SG ^{(\omega _0)}_{r,\rho}(\rr {2d})$ which is
supported outside $\rr d\times \Omega$. Similarly, the pair $(\cB ,\cC )$ is called 
\emph{weakly-II $\SG$-ordered (with respect to
$(r,\rho ,\omega _0,\varphi ,\Omega )$)}, when the mapping
$$
\op_\varphi^*(b)\, :\, \cC \to \cB
$$
is continuous for every $b\in \SG ^{(\omega _0)}_{r,\rho}(\rr {2d})$
which is supported outside $\Omega\times\rr d$. Furthermore,
$(\cB_1, \cC_1, \cB_2, \cC_2)$ are called \emph{$\SG$-ordered
(with respect to $r$, $\rho$, $\omega _1$, $\omega_2$, $\varphi$,
and $\Omega$)}, when $(\cB_1,\cC_1)$ is a weakly-I 
$\SG$-ordered pair with respect to
$(r,\rho ,\omega _1,\varphi ,\Omega )$, and
$(\cB_2,\cC_2)$ is a weakly-II 
$\SG$-ordered pair with respect to
$(r,\rho ,\omega _2,\varphi ,\Omega )$.
\end{defn}

\par

\begin{rem}\label{remSGadm}
Let $\sigma _1$, $\sigma _2$, $p$, $m$ and $\mu$ be the same as in
Theorem \ref{FIOSobcont}. Then it follows that
$(H^p_{\sigma _1,\sigma _2},H^p_{\sigma_1-\mu,\sigma_2-m})$
are weakly-I $\SG$-ordered with respect to
$(r,\rho ,\omega _0,\varphi ,\Omega )$, when
$$
\omega _0(x,\xi) = \eabs x^{m}\eabs \xi ^{\mu}\quad
\text{and}\quad \Omega = B_r(0).
$$

\par

\noindent
Furthermore, if $p$, $m$, $\mu$
and $\omega _j$, $j=0,1,2$ are the same as in Theorem
\ref{FIOmodcont}, then it follows that
$(M^p_{(\omega _1)},M^p_{(\omega _2)})$ are weakly-I
$\SG$-ordered with respect to $(r,\rho ,\omega _0,\varphi ,\emptyset )$.
\end{rem}

\par

\section{Propagation results for global Wave-front Sets and
generalised FIOs of $\SG$ type}\label{sec4}

\par

%
%
We first recall the definition of global wave-front sets with respect
to modulation spaces, given in \cite{CJT2}. The content of
Subsection \ref{subs:4.1} again comes from \cite{CJT3}.
In Subsection \ref{subs:4.2} we prove our main results about the propagation
of singularities in the $\SG$ context, under the action of the Fourier integral
operators described above. 

\par

\subsection{Global Wave-front Sets}\label{subs:4.1}
Here we recall the definition given in \cite{CJT2} of global
wave-front sets for temperate distributions with respect to Banach
or Fr\'echet spaces and state some of their properties (see also
\cite{CJT3}). First of all, we recall the definitions of set of
characteristic points that we use in this setting.
Remember that if $a\in \SG^{(\omega _0 )}_{r,\rho}(\rr {2d})$, then
$$
|a(x,\xi )|\lesssim \omega _0 (x,\xi ).
$$
On the other hand, $a$ is invertible, in the sense that $1/a$ is a symbol in
$\SG^{(1/\omega _0 )}_{r,\rho}(\rr {2d})$, if and only if
\begin{equation}\label{invcond}
\omega _0 (x,\xi )\lesssim |a(x,\xi )|.
\end{equation}
We need to deal with the situations where \eqref{invcond} holds only in certain
(conic-shaped) subset of $\rr d \times \rr d$. Here we let $\Omega _m$,
$m=1,2,3$, be the sets
\begin{equation}\label{omegasets}
\begin{aligned}
\Omega _1= \rr d\times (&\rr d\setminus 0),\qquad
\Omega _2 = (\rr d \setminus 0)\times \rr d,
\\[1ex]
\Omega _3 &= (\rr d \setminus 0)\times (\rr d \setminus 0),
\end{aligned}
\end{equation}

\par

\begin{defn}\label{defchar}
Let $r,\rho \ge 0$, $\omega _0\in \mathscr P _{r,\rho} (\rr
{2d})$, $\Omega _m$, $m=1,2,3$ be as in \eqref{omegasets},
and let $a\in \SG ^{(\omega _0)}_{r,\rho}(\rr {2d})$.

\medspace

\begin{enumerate}

\item $a$ is called \emph{locally} or \emph{type-$1$ invertible} 
with respect to $\omega _0$ at the
point $(x_0,\xi_0)\in \Omega _1$, if
there exist a neighbourhood $X$ of $x_0$, an open conical
neighbourhood $\Gamma$ of $\xi _0$  and a positive constant $R$
such that \eqref{invcond} holds for $x\in X$, $\xi\in \Gamma$ and
$|\xi|\ge R$.

\vrum

\item  $a$ is called
\emph{Fourier-locally} or \emph {type-$2$ invertible} with respect to
$\omega _0$ at the point $(x_0,\xi_0)\in \Omega _2$, if
there exist an open conical neighbourhood $\Gamma$  of $x_0$, a
neighbourhood $X$ of $\xi _0$ and a positive constant $R$ such
that \eqref{invcond} holds for $x\in \Gamma$, $|x|\ge R$ and  $\xi\in
X$.

\vrum

\item  $a$ is called
\emph{oscillating} or \emph{type-$3$ invertible} with respect to
$\omega _0$ at the point $(x_0,\xi_0)\in \Omega _3$, if
there exist open conical neighbourhoods $\Gamma _1$  of $x_0$ and
$\Gamma _2$ of $\xi _0$, and a positive constant $R$
such that \eqref{invcond} holds for $x\in \Gamma _1$, $|x|\ge R$,
$\xi \in \Gamma _2$ and $|\xi |\ge R$.

\end{enumerate}

\par

If $m\in \{ 1,2,3\}$ and  $a$ is \emph{not} type-$m$ invertible
with respect to $\omega _0$ at $(x_0,\xi_0)\in \Omega _m$,
then $(x_0,\xi_0)$ is called \emph{type-$m$ characteristic} for $a$ with
respect to $\omega _0$. The set of type-$m$
characteristic points for $a$ with respect to $\omega _0$ is denoted by
$\Char _{(\omega _0)}^m(a)$.

\par

The \emph{(global) set of characteristic points} (the characteristic set), for a
symbol $a\in \SG^{(\omega _0)}_{r,\rho}(\rr {2d})$ with respect to $
\omega _0$ is defined as
$$
\Char (a)=\Char _{(\omega _0)}(a)=\Char ^1_{(\omega
_0)}(a)\bigcup\Char ^2 _{(\omega _0)}(a)\bigcup\Char ^3_{(\omega
_0)}(a).
$$
\end{defn}

\par

\begin{rem}\label{psiinvremark}
Let $X\subseteq \rr d$ be open and $\Gamma ,\Gamma _1,\Gamma _2
\subseteq \rr d\back 0$
be open cones. Then the following is true.
\begin{enumerate}
\item if $x_0\in X$, $\xi _0\in \Gamma$, $\fy \in \mathscr
C _{x _0}(X)$ and $\psi \in \mathscr C ^{\dir} _{\xi _0}(\Gamma )$,
then $c_1=\fy \otimes \psi$ belongs to $\SG^{0,0}_{1,1}(\rr {2d})$,
and is type-$1$ invertible at $(x_0,\xi _0)$;

\vrum

\item if $x_0\in \Gamma$, $\xi _0\in X$, $\psi \in \mathscr C ^{\dir}
_{x_0}(\Gamma )$ and $\fy \in \mathscr C _{\xi  _0}(X)$,
then $c_2=\fy \otimes \psi$ belongs to $\SG^{0,0}_{1,1}(\rr {2d})$,
and is type-$2$ invertible at $(x_0,\xi _0)$;

\vrum

\item if $x_0\in \Gamma _1$, $\xi _0\in \Gamma _2$, $\psi _1\in
\mathscr C ^{\dir} _{x_0}(\Gamma _1)$ and $\psi _2\in
\mathscr C ^{\dir} _{\xi _0}(\Gamma _2)$, then $c_3=\psi _1 \otimes
\psi _2$ belongs to $\SG^{0,0}_{1,1}(\rr {2d})$, and is type-$3$
invertible at $(x_0,\xi _0)$.
\end{enumerate}
\end{rem}

\par

\begin{rem}
In the case $\omega _0=1$ we exclude the phrase ``with respect to
$\omega _0$" in Definition \ref{defchar}. For example, $a\in
\SG^{0,0}_{r,\rho}(\rr{2d})$ is \emph{type-$1$ invertible} at
$(x_0,\xi_0) \in \rr d\times (\rr d\back 0)$ if $(x_0,\xi _0)\notin
\Char ^1 _{(\omega _0)}(a)$ with $\omega _0=1$. This means that there
exist a neighbourhood $X$ of $x_0$, an open conical
neighbourhood $\Gamma$ of $\xi_0$ and $R>0$ such that
\eqref{invcond} holds for $\omega _0=1$, $x\in X$ and $\xi\in \Gamma$
satisfies $|\xi|\geq R$.
\end{rem}
%
%
%
%

\par

In the next definition we introduce different classes of cutoff
functions (see also Definition 1.9 in \cite{CJT1}).

\par

\begin{defn}\label{cuttdef}
Let $X\subseteq \rr d$ be open, $\Gamma \subseteq \rr
d\setminus 0$ be an open cone, $x_0\in X$ and let
$\xi _0\in \Gamma$.

\begin{enumerate}
\item A smooth function $\fy$ on $\rr d$ is called a \emph{cutoff
(function)} with respect to $x_0$ and $X$, if $0\le \fy \le 1$, $\fy \in
C_0^\infty (X)$ and $\fy =1$ in an open neighbourhood of $x_0$. The
set of cutoffs with respect to $x_0$ and $X$ is denoted by
$\mathscr C_{x_0}(X)$ or $\mathscr C_{x_0}$.

\vrum

\item  A smooth function $\psi$ on $\rr d$ is called a
\emph{directional cutoff (function)} with respect to $\xi_0$ and
$\Gamma$, if there is a constant $R>0$ and open conical neighbourhood
$\Gamma _1\subseteq \Gamma$ of $\xi _0$ such that the following is
true:
\begin{itemize}
\item $0\le \psi \le 1$ and $\supp \psi \subseteq \Gamma$;

\vrum

\item  $\psi (t\xi )=\psi (\xi )$ when $t\ge 1$ and $|\xi |\ge R$;

\vrum

\item $\psi (\xi )=1$ when $\xi \in \Gamma _1$ and $|\xi |\ge R$.
\end{itemize}

\par

The set of directional cutoffs with respect to $\xi _0$ and
$\Gamma$ is denoted by $\mathscr C^\dir  _{\xi _0}(\Gamma )$ or
$\mathscr C^\dir  _{\xi _0}$.
\end{enumerate}
\end{defn}

\par

The next proposition shows that $\op _t(a)$ for
$t\in \mathbf R$ satisfies
convenient invertibility properties of the form
\begin{equation}\label{locinvop}
\op _t(a)\op _t(b) = \op _t(c) + \op _t(h),
\end{equation}
outside the set of characteristic
points for a symbol $a$. Here $\op _t(b)$, $\op _t(c)$ and $\op _t(h)$ have
the roles of ``local inverse'', ``local identity'' and smoothing operators
respectively. From these statements it also follows that our set of
characteristic points in Definition \ref{defchar} are related to those
in \cite{CoMa,Ho1}. 
We let $\mathbb I_m$,
$m=1,2,3$, be the sets
\begin{equation}\label{isets}
\mathbb I_1  \equiv [0,1]\times (0,1],\,
\mathbb I_2 \equiv (0,1]\times [0,1],\,
\mathbb I_3 \equiv  (0,1]\times (0,1] =
\mathbb I_1\cap \mathbb I_2.
\end{equation}
which will be useful in the sequel.

\par

\begin{prop}\label{charequiv}
Let $m\in \{1,2,3\}$, $(r,\rho )\in \mathbb I_m$, $\omega _0\in
\mathscr P_{r,\rho}(\rr {2d})$ and let $a\in \SG ^{(\omega _0)}_{r,\rho }(\rr
{2d})$. Also let $\Omega _m$ be as in \eqref{omegasets}, $(x_0,\xi _0)\in
\Omega _m$, and let $(r_0,\rho _0)$ be equal to $(r,0)$, $ (0,\rho ) $
and $(r,\rho )$ when $m$ is equal to $1$, $2$ and $3$, respectively.
Then the following conditions are equivalent:

\medspace

\begin{enumerate}

\item $(x_0,\xi _0)\notin \Char _{(\omega _0)}^m (a)$;

\vrum

\item there is an element $c\in \SG^{0,0}_{r,\rho}$ which is
type-$m$ invertible at $(x_0,\xi _0)$, and an element $b\in
\SG^{(1/\omega _0)}_{r,\rho}$ such that $ab=c$;

\vrum

\item \eqref{locinvop} holds for some $c\in \SG^{0,0}_{r,\rho}$ which is
type-$m$ invertible at $(x_0,\xi _0)$, and some elements  $h\in
\SG^{-r_0,-\rho _0}_{r,\rho}$ and $b\in \SG^{(1/\omega _0)}_{r,\rho}$;

\vrum

\item \eqref{locinvop} holds for some $c_m\in \SG^{0,0}_{r,\rho}$
in Remark \ref{psiinvremark} which is
type-$m$ invertible at $(x_0,\xi _0)$, and some elements  $h$ and $b\in
\SG^{(1/\omega _0)}_{r,\rho}$, where $h\in \mathscr S$ when $m\in
\{ 1,3\}$ and $h\in \SG ^{-\infty ,0}$ when $m=2$.

\par

Furthermore, if $t=0$, then the supports of $b$ and $h$ can be chosen to be
contained in $X\times \rr d$ when $m=1$, in $\Gamma \times \rr d$ when
$m=2$, and in $\Gamma _1\times \rr d$ when $m=3$.
\end{enumerate}
\end{prop}

\par

We can now introduce the complements of the wave-front sets.
More precisely, let $\Omega _m$, $m\in \{ 1,2,3\}$, be given by
\eqref{omegasets}, $\cB$ be a Banach or Fr{\'e}chet space such
that $\mathscr S(\rr d)\subseteq \cB \subseteq \mathscr S'(\rr d)$,
and let $f\in \mathscr S'(\rr d)$. Then the point $(x_0,\xi _0)\in
\Omega _m$ is called \emph{type-$m$} regular for $f$ with
respect to $\cB$, if
\begin{equation}\label{ImageOpcm}
\op (c_m)f\in \cB ,
\end{equation}
for some $c_m$ in Remark \ref{psiinvremark}. The set of all type-$m$
regular points for $f$ with respect to $\cB$, is denoted by
$\Theta ^m_{\cB}(f)$.

\par

\begin{defn}\label{def:wfsMB}
Let $m\in \{ 1,2,3\}$, $\Omega _m$ be as in \eqref{omegasets}, and
let $\cB$ be a Banach or Fr\'echet space such that
$\mathscr S(\rr d)\subseteq \cB \subset \mathscr S^\prime(\rr d)$.
\begin{enumerate}
\item the \emph{type-$m$ wave-front set} of $f\in \mathscr S'(\rr d)$ with
respect to $\cB$ is the complement of $\Theta ^m_{\cB}(f)$ in $\Omega
_m$, and is denoted by $\WFF ^m_{\cB}(f)$;

\vrum

\item the \emph{global wave-front set} $\WFgB(f)\subseteq (\rr
d\times \rr d)\back 0$ is the set
\begin{equation*}
\WFgB(f) \equiv \WFF ^1_\cB (f) \bigcup \WFF ^2_\cB (f) \bigcup \WFF ^3_\cB (f).
\end{equation*}
\end{enumerate}
\end{defn}

\par

The sets $\WFF ^1_{\cB}(f)$, $\WFF ^2_{\cB}(f)$ and $\WFF ^3_{\cB}(f)$
in Definition \ref{def:wfsMB}, are also called the \emph{local},
\emph{Fourier-local} and \emph{oscillating} wave-front set of $f$ with
respect to $\cB$.

\par

\begin{rem}\label{rem:wfcon}
        Let $\Omega _m$, $m=1,2,3$ be the same as in \eqref{omegasets}. 
	\begin{enumerate}
		\item If $\Omega \subseteq \Omega _1$, and
		$(x_0,\xi_0)\in \Omega \ \Longleftrightarrow \ 
			(x_0,\sigma \xi_0)\in \Omega$ for $\sigma \ge 1$,
			then $\Omega$ is called \emph{$1$-conical};
		
		\vrum

		\item If $\Omega \subseteq \Omega _2$, and $(x_0,\xi_0)\in
		\Omega \ \Longleftrightarrow \ 
			(sx_0,\xi_0)\in\Theta^2_{\cB}(f)$ for $s \ge 1$,
			then $\Omega$ is called \emph{$2$-conical};
		
		\vrum

		\item If $\Omega \subseteq \Omega _3$, and $(x_0,\xi_0)
		\in \Omega\ \Longleftrightarrow \ 
			(sx_0,\sigma\xi_0)\in \Omega$ for $s,\sigma \ge 1$,
		then $\Omega$ is called \emph{$3$-conical}.
	\end{enumerate}
        By \eqref{ImageOpcm} and the paragraph before Definition \ref{def:wfsMB},
	it follows that if $m=1,2,3$, then $\Theta^m_\cB (f)$ is $m$-conical.
	The same holds for $\WFF^m_\cB(f)$, $m=1,2,3$, 
	by Definition \ref{def:wfsMB}, noticing that, for any $x_0\in
	\rr{r} \setminus \{0\}$, any open cone $\Gamma\ni x_0$,
	and any $s>0$, $\mathscr C ^{\dir} _{x_0}(\Gamma )
	=\mathscr C ^{\dir} _{s x_0}(\Gamma )$.
	For any $R>0$ and $m\in \{1,2,3\}$, we set
\begin{gather*}
\Omega_{1,R} \equiv \sets {(x,\xi )\in \Omega _1}{|\xi |\ge R},
\quad
\Omega_{2,R} \equiv \sets {(x,\xi )\in \Omega _2}{|x |\ge R},
\\[1ex]
\Omega_{3,R} \equiv \sets {(x,\xi )\in \Omega _3}{|x|, |\xi |\ge R}
\end{gather*}
	Evidently, $\Omega_m^R$ is $m$-conical for every $m\in \{ 1,2,3\}$.
\end{rem}

\par

From now on we assume that $\cB$ in Definition \ref{def:wfsMB} is
$\SG$-admissible, and recall
that Sobolev-Kato spaces and, more generally, modulation spaces, and
$\mathscr S(\rr d)$ are $\SG$-admissible, see \cite{CJT2, CJT3}.

\par

The next result describes the relation between ``regularity with respect to
$\cB$\,'' of temperate distributions and global wave-front sets.

\par

\begin{prop}\label{mainthm4}
Let $\cB$ be $\SG$-admissible, and let $f\in\cS^\prime (\rr d)$. Then
\begin{equation*}
f\in\cB \quad \Longleftrightarrow \quad
\WFgB(f)=\emptyset.
\end{equation*}
\end{prop}

\par

For the sake of completeness, we recall that microlocality and microellipticity
hold for these global wave-front sets and pseudo-differential operators in
$\op(\SG^{(\omega_0)}_{r,\rho})$, see \cite{CJT2}.
This implies that 
operators which are elliptic with respect to $\omega _0\in \mathscr
P_{\rho,\delta}(\rr {2d})$ when $0< r ,\rho \le 1$ preserve the global wave-front set
of temperate distributions. 
The next result is an immediate corollary
of microlocality and microellipticity
for operators in $\op(\SG^{(\omega_0)}_{r,\rho})$:

\par

\begin{prop}\label{hypoellthm}
Let $m\in \{ 1,2,3\}$, $(r,\rho )\in \mathbb I_m$, $t\in \mathbf R$, 
$\omega
_0\in  \mathscr P_{r,\rho}(\rr {2d})$, $a\in \SG^{(\omega _0)}_{r,\rho} 
(\rr {2d})$ be $\SG$-elliptic
with respect to $\omega _0$
and let $f\in \cS '(\rr d)$. Moreover, let $(\cB , \cC )$ be a 
$\SG$-ordered pair with respect to $\omega _0$.
Then
\begin{align*}
\WFF ^{m} _{\mathcal C} (\op _t(a)f) &=
\WFF ^{m}_{\mathcal B} (f).
\end{align*}
\end{prop}

\par

\subsection{Action of generalised FIOs of $\SG$ type on
global Wave-front Sets}\label{subs:4.2}
We let $\phi$ be the canonical transformation
of $T^\star \rr{d}$ into itself generated by the phase function
$\varphi \in \Phr$. This means that $\phi = (\phi_1,\phi_2)$ is
the smooth function on $T^*\rr d$ into itself, defined by the
relations
\begin{equation}
(x,\xi )=\phi (y,\eta )
\quad \Longleftrightarrow \quad
\label{eq:trsympl}
  \begin{cases}
     y  = \varphi^\prime_{\xi} (x,\eta) =  \varphi^\prime_{\eta} (x,\eta),
   \\[1ex]
     \xi = \varphi^\prime_x (x, \eta),
  \end{cases}
\end{equation}
As we have seen in Subsection \ref{sec:2.5}, such transformations
appear in the Egorov's theorem, through which we prove
Theorems \ref{thm:3.22} and Corollaries \ref{cor:FIOWFSell} and \ref{cor:FIOWFSellter} 
below. This justifies the following definition
of admissibility of phase functions. Namely, the latter are required to
generate transformations of the type \eqref{eq:trsympl} which ``preserve
the shape'' of the different kinds of neighborhoods 
appearing in the Definition \ref{defchar} of set of characteristic points.
\begin{defn}\label{def:gadm}
Let $\varphi\in\Phr$ and let $\phi$ be the canonical transformation
\eqref{eq:trsympl}, generated by $\varphi$. Let $m\in \{1,2,3\}$ and
$\Omega _m$ be as in \eqref{omegasets}.
%
\begin{enumerate}
	\item $\varphi$ is called \emph{$1$-admissible} at $(y_0,\eta_0) \in
	\Omega _1$ if, for every $1$-cone $X\times \Gamma$ containing
	$\phi (y_0,\eta_0)$ and $r>0$, there is a $1$-cone $Y\times \Gamma _0$
	containing $(y_0,\eta _0)$ and $R>0$ such that
	         \begin{equation*}
	         \phi (y,\eta ) \in (X\times \Gamma)\bigcap \Omega _{1,r}
	         \quad \text{when}\quad
	         (y,\eta ) \in (Y\times \Gamma _0)\bigcap \Omega _{1,R}\text ;
	         \end{equation*}

\vrum

	\item $\varphi$ is called \emph{$2$-admissible} at $(y_0,\eta_0) \in
	\Omega _2$ if, for every $2$-cone $\Gamma \times X$ containing
	$\phi (y_0,\eta_0)$ and $r>0$, there is a $2$-cone $\Gamma _0\times Y$
	containing $(y_0,\eta _0)$ and $R>0$ such that
	         \begin{equation*}
	         \phi (y,\eta ) \in (\Gamma \times X)\bigcap \Omega _{2,r}
	         \quad \text{when}\quad
	         (y,\eta ) \in (\Gamma _0\times Y)\bigcap \Omega _{2,R}\text ;
	         \end{equation*}
	         
\vrum
	         
	\item $\varphi$ is called \emph{$3$-admissible} at $(y_0,\eta_0) \in
		\Omega _3$ if, for every $3$-cone $\Gamma _1\times \Gamma _2$
		containing $\phi (y_0,\eta_0)$ and $r>0$, there is a $3$-cone
		$\Gamma _{0,1}\times \Gamma _{0,2}$
		containing $(y_0,\eta _0)$ and $R>0$ such that
	         \begin{equation*}
	         \phi (y,\eta ) \in (\Gamma _1\times \Gamma _2)
	         \bigcap \Omega _{3,r}
	         \quad \text{when}\quad
	         (y,\eta ) \in (\Gamma _{0,1}\times \Gamma _{0,2})\bigcap
	         \Omega _{3,R}\text .
	         \end{equation*}
\end{enumerate} 

\par

Furthermore, $\varphi$ is called \emph{$m$-admissible} if it is $m$-admissible
at all points $(y,\eta )\in \Omega _m$, and 
$\varphi$ is called \emph{admissible} if it is $m$-admissible for all $m=1,2,3$.
\end{defn} 

\par

\begin{rem}\label{rem:invtransf}
Notice that the inverse transformation $\phi^{-1}$ is defined as in
\eqref{eq:trsympl}, by exchanging the role of $(x,\xi)$ and $(y,\eta)$.
If $\varphi$ is $m$-admissible, $m=1,2,3$, then
both $\phi$ and $\phi^{-1}$ satisfy the corresponding property in
Definition \ref{def:gadm}.
\end{rem}

\par

\begin{rem}\label{rem:transfchar}
Let $\varphi$ be $m$-admissible, $m=1,2,3$, and let $\omega _0\in
\mathscr P_{1,1}(\rr {2d})$ be invariant with respect to the canonical transformation
\eqref{eq:trsympl}.
For any $a\in \SG ^{(\omega _0)} _{r,\rho }(\rr{2d})$, setting $\widetilde{\omega}_0=
\omega_0\circ\phi$, we have
\[
	(y_0,\eta_0)\in\Char^m_{(\widetilde{\omega}_0)}(a\circ\phi) 
	\quad \Longleftrightarrow \quad 
	(x_0,\xi_0)=\phi(y_0,\eta_0)\in\Char^m_{(\omega_0)}(a).
\]
By Remark \ref{rem:invtransf}, similar properties hold with
$\phi^{-1}$ in place of $\phi$.
\end{rem}

\par

\begin{rem}\label{rem:transfcharbis}
Let $\fy _m$, $m=1,2,3$, be phase functions such that
\begin{itemize}
\item $\xi \mapsto \fy _1(x,\xi )$ is homogeneous of order $1$
for large $|\xi |$;

\vrum

\item $x \mapsto \fy _1(x,\xi )$ is homogeneous of order $1$
for large $|x|$;

\vrum

\item $x \mapsto \fy _1(x,\xi )$ and $\xi \mapsto \fy _1(x,\xi )$ are
homogeneous of order $1$ for large $|x|$ and $|\xi |$.
\end{itemize}

\par

Such phase functions are common in the literature.
An example of admissible phase functions, which is not necessarily
homogeneous, is given by the so called \emph{$\SG$-classical phase
functions}. Families of such objects, smoothly depending on  a
parameter $t\in[-T,T]$, $T>0$, are obtained by solving Cauchy
problems associated with classical $\SG$-hyperbolic systems
with diagonal principal part.

In fact, omitting the dependence on the \emph{time variable $t$},
an $\SG$-classical phase functions $\varphi$ admits expansions in
terms which are homogeneous with respect to
$x$, respectively $\xi$, satisfying suitable compatibility relations, see, e.g.,
\cite{CoMa, CoPa}. In particular, $\varphi$ admits a principal
symbol, given by a triple $(\varphi_1,\varphi_2,\varphi_3)$, that is,
it can be written as
\begin{equation}
	\label{eq:classsymp}
	\begin{aligned}
	\varphi (x,\xi )&=\chi (\xi )\,\varphi _1(x,\xi )+\chi (x)(\varphi _2(x,\xi )
	-\chi (\xi )\,\varphi _3(x,\xi ))
	\\
	&\mod \SG^{0,0}_{1,1}(\rr{2d}).
	\end{aligned}
\end{equation}
In \eqref{eq:classsymp}, $\chi$ is a $0$-excision function, while 
$\chi (\xi )\,\varphi _1(x,\xi )$, $\chi (x)\,\varphi_2(x,\xi )$, $\chi (\xi )\,
\chi (x)\,\varphi _3(x,\xi )\in \SG^{1,1}_{1,1}(\rr{2d})$, where
$\varphi_1$ is $1$-homogeneous with respect to the variable $\xi$,
$\varphi_2$ is $1$-homogeneous with respect to the variable $x$,
and $\varphi_3$ is $1$-homogeneous with respect to each one of
the variables $x,\xi$. Observe that then
\begin{equation}
	\label{eq:classphf}
	\begin{aligned}
		\varphi(x,\xi)&=\chi(\xi)\,\varphi_1(x,\xi)\mod\SG^{1,0}_{1,1}(\rr{2d}),
		\\[1ex]
		\varphi(x,\xi)&=\chi(x)\,\varphi_2(x,\xi)\mod\SG^{0,1}_{1,1}(\rr{2d}),
		\\[1ex]
		\varphi(x,\xi)&=\chi(x)\,\chi(\xi)\,\varphi_3(x,\xi)\mod
		\SG^{0,1}_{1,1}(\rr{2d})+\SG^{1,0}_{1,1}(\rr{2d}).
	\end{aligned}
\end{equation}
The homogeneity of the leading terms in \eqref{eq:classphf} implies, in particular,
\begin{equation}
	\label{eq:classphfbis}
	\begin{aligned}
		\varphi^\prime_x(x,\xi)&=
		|\xi|\left[\varphi^\prime_{1,x}\left(x,\dfrac{\xi}{|\xi|}\right
		)+|\xi|^{-1}r_1(x,\xi)\right]
		\text{ for $|\xi|>R$},
	\\[1ex]
		\varphi^\prime_\xi(x,\xi)&=
		|x|\left[\varphi^\prime_{2,\xi}\left(\dfrac{x}{|x|},\xi\right)+|x|^{-1}r_2(x,\xi)\right]
		\text{ for $|x|>R$},
	\\[1ex]
		\varphi^\prime_x(x,\xi)&=
		|\xi|\left[\varphi^\prime_{3,x}\left(\dfrac{x}{|x|},\dfrac{\xi}{|\xi|}\right)
		+
		           |\xi|^{-1}(r_{31}(x,\xi)+s_{31}(x,\xi))\right]
		           \\
		&\text{ for $|x|,|\xi|>R$},
	\\[1ex]
		\varphi^\prime_\xi(x,\xi)&=
		|x|\left[\varphi^\prime_{3,\xi}\left(\dfrac{x}{|x|},\dfrac{\xi}{|\xi|}\right)
		+
		          |x|^{-1}(r_{32}(x,\xi)+s_{32}(x,\xi))\right]
		          \\
		&\text{ for $|x|,|\xi|>R$},
	\end{aligned}
\end{equation}
with $r_1,r_2,r_{31},r_{32}\in\SG^{0,0}_{1,1}(\rr{2d})$, $s_{31}\in\SG^{-1,1}_{1,1}(\rr{2d})$,
$s_{32}\in\SG^{1,-1}_{1,1}(\rr{2d})$. By the properties of generalised $\SG$ symbols and \eqref{eq:classphfbis} it is possible to prove that all the conditions in Definition \ref{def:gadm} are fulfilled.
\end{rem}

\par

We can now state the first of our main results concerning the propagation of (global)
singularities under the action of the generalised $\SG$ FIOs.

\par

\begin{thm}
\label{thm:3.22}
Let $\varphi\in\Phr$ be $m$-admissible, $m\in\{1,2,3 \}$, 
$\omega_0\in \mathscr P_{r,\rho}(\rr {2d})$, 
$a\in\SG^{(\omega_0)}_{r,\rho}(\rr{2d})$, supported
outside $\rr{d}\times\Omega$,
$\Omega\subseteq\rr{d}$ open. Assume that $\omega_0$ is $\phi$-invariant, 
where $\phi$ is as in Theorem \ref{thm:3.8}.
Assume also that $a$ is $\SG$-elliptic
and $(\cB ,\cC )$ is a weakly-I $\SG$-ordered pair with respect to
$(r,\rho ,\omega _0,\varphi ,\Omega )$.
Then,
\begin{equation}
\label{eq:3.72.8}
\WFF^{m}_\cC(\op_\varphi(a)f) = \phi(\WFF^{m}_\cB(f)),\quad f
\in \mathscr{S}^\prime(\rr{d}),
\end{equation}
where $\phi$ is the canonical transformation \eqref{eq:trsympl}, generated by $\varphi$.
\end{thm}

\par

\begin{proof}
We only prove the result for $m=3$. The other cases follow
by similar arguments and are left for the reader.
Let $(y_{0}, \eta_{0})=\phi^{-1}(x_{0}, \xi_{0})\in\Theta ^m_\cB (f)$,
$m\in\{1,2,3\}$, and  let $c_m \in \SG^{0,0}_{1,1}$ be a symbol as
in \eqref{ImageOpcm} and Remark \ref{psiinvremark}
such that $\op(c_m) u \in \cB$. Recalling Remark \ref{rem:trivweight}, the weight 
$\omega(x,\xi)=\vartheta_{0,0}(x,\xi)=1\in\mathscr{P}_{1,1}$ is invariant with respect to any
$\SG$ diffeomorphism with $\SG^0$ parameter dependence.
Let $A=\op_\varphi(a)$, $C_m=\op(c_m)$, and let $B$ be a
parametrix for $A$. 
Then for some $q_m$ we have
$$
Q_m = A\circ C_m \circ B,\qquad Q_m=\op (q_m),
$$
or equivalently,
$$
Q_m\circ A = A\circ C_m\mod \op (\SG ^{-\infty,-\infty)}).
$$

\par

By Theorem \ref{thm:3.21ell} and \eqref{eq:trsympl}, we have
$q_m = c_m \circ \phi^{-1}\mod \SG ^{-1,-1}_{1,1}$, which implies
$q_m\in\SG^{0,0}_{1,1}$. Then, 
$(x_{0}, \xi_{0})\in \Theta ^m_{\cC} (Af)$, since $Q_m (A f) \equiv A(C_mf)
\in \cC$, by the hypotheses on $(\cB,\cC)$. This means that
\begin{equation}
\label{eq:3.72.9}
\phi(\Theta ^{m}_\cB(f)) \subseteq \Theta ^{m} _\cC(Af).
\end{equation}
Complementing \eqref{eq:3.72.9} with respect to $\Omega_m$,
repeating a similar argument starting
from $Af$, recalling Remark \ref{rem:transfchar} and that $\phi$
is a diffeomorphism,
we finally obtain \eqref{eq:3.72.8}.
\end{proof}

\par

The next result is proved in a similar fashion. In fact, with 
notation analogous to the one used in the proof of Theorem
\ref{thm:3.22}, denoting $B=\op ^*_\varphi(b)$,
we have that $Q_m=B\circ C_m\circ B^{-1}$ satisfies
$\Sym{Q_m}=c_m\circ\phi$ modulo lower order terms.
It is then enough to recall Remark \ref{rem:invtransf}.
Of course, when one deals with $\SG$-ordered spaces
$(\cB_1,\cC_1,\cB_2, \cC_2)$, both \eqref{eq:3.72.8}
and \eqref{eq:ellwfbis} hold, as stated in Corollary
\ref{cor:FIOWFSellter}.

\begin{cor}
\label{cor:FIOWFSell}
Let $\varphi\in\Phr$ be $m$-admissible, $m\in\{1,2,3\}$, 
$\omega_0\in \mathscr P_{1,1}(\rr {2d})$, 
$b\in\SG^{(\omega_0)}_{1,1}(\rr{2d})$, supported outside $\Omega\times\rr{d}$,
$\Omega\subseteq\rr{d}$ open. Assume that $\omega_0$ is $\phi^{-1}$-invariant,
where $\phi$ is as  in Theorem \ref{thm:3.8}. Assume also that $b$ is $\SG$-elliptic
and that $(\cB,\cC)$ is a weakly-II $\SG$-ordered pair with respect to
$(r,\rho ,\omega _0,\varphi ,\Omega )$. Then, 
\begin{equation}\label{eq:ellwfbis}
	\WFF^{m}_{\cB}(\op_\varphi^*(b)f) = \phi^{-1}(\WFF^{m}_{\cC}(f)),\quad f
\in \mathscr{S}^\prime(\rr{d}),
\end{equation}
with the inverse $\phi^{-1}$ of the canonical transformation \eqref{eq:trsympl}.
\end{cor}

\begin{cor}
\label{cor:FIOWFSellter}
Let $\varphi\in\Phr$ be $m$-admissible, $m\in\{1,2,3\}$, 
$\omega_1,\omega_2\in \mathscr P_{1,1}(\rr {2d})$. Moreover, let
$a\in\SG^{(\omega_1)}_{r,\rho}(\rr{2d})$,
$b\in\SG^{(\omega_2)}_{r,\rho}(\rr{2d})$, with $a$ supported outside $\rr{d}\times\Omega$,
$b$ supported outside $\Omega\times\rr{d}$, respectively, where
$\Omega\subseteq\rr{d}$ is open. 
Assume that $a$ and $b$ are $\SG$-elliptic
and that $(\cB_1,\cC_1,\cB_2,\cC_2)$ are $\SG$-ordered with respect to
$$
r,\  \rho,\ \omega _1 ,\ \omega _2 ,\ \varphi \ \text{and}\ \Omega  .
$$
Then, provided that $\omega_1$ and $\omega_2$ satisfy the invariance properties
required in Theorem \ref{thm:3.22} and Corollary \ref{cor:FIOWFSell}, respectively,
for any $f \in \cS^\prime(\rr{d})$, we have
\[
\WFF^{m}_{\cC_1}(\op_\varphi(a)f) = \phi(\WFF^{m}_{\cB_1}(f)),
\]
and
\[
	\WFF^{m}_{\cB_2}(\op_\varphi^*(b)f) = \phi^{-1}(\WFF^{m}_{\cC_2}(f)),
\]
with the canonical transformation $\phi$ in \eqref{eq:trsympl} and its inverse $\phi^{-1}$ .
\end{cor}

\par

The next result generalizes Theorem \ref{thm:3.22} and Corollaries
\ref{cor:FIOWFSell} and \ref{cor:FIOWFSellter} to the case where
the involved amplitudes are not $\SG$-elliptic. In such a situation,
the set of admissible phase functions needs to be slightly restricted,
similarly to the calculus of Fourier integral operators developed in
\cite{Ku}. Such restriction is not very harmful, since the phase
functions we are mostly interested in are those appearing in the
next Subsection \ref{subs:4.3}, and it can be proved that they fulfill
\eqref{eq:varphitau} below if a sufficiently small ``time interval''
$J^\prime=[-T^\prime,T^\prime]$ is chosen. This can be easily
verified by checking the technique of solution of the involved
eikonal equations, see, e.g., \cite{Ku, coriasco2, CoPa, CoRo}.
Here the symbols satisfy
\begin{equation}\label{eq:abSupport}
\supp a\subseteq \rr d\times \Omega ^\complement ,
\quad
\supp b\subseteq \Omega ^\complement \times \rr d
\end{equation}
for suitable open set $\Omega $, where $\Omega ^\complement$ 
equals $\rr d\setminus \Omega$, and the phase function satisfies
\begin{equation}\label{eq:varphitau}
\begin{gathered}
	|\norm{x}^{-1+|\alpha|}\norm{\xi}^{-1+|\beta|}D^\alpha_xD^\beta_\xi
	\kappa (x,\xi)|\le\tau,\quad
	x,\xi \in \rr d,\ |\alpha+\beta|\le2,
\\
\text{where}\quad
\kappa (x,\xi)=\varphi(x,\xi)-\langle x,\xi\rangle\in\SG^{1,1}_{1,1}(\rr{2d}).
\end{gathered}
\end{equation}

\par

\begin{thm}
\label{thm:3.23}
Let $\varphi \in \Phr$ be $m$-admissible, $m\in\{1,2,3\}$, and fulfill \eqref{eq:varphitau}
for a fixed $\tau \in (0,1)$.
Also let $\omega _1\in \mathscr P_{r,\rho}(\rr {2d})$, $r,\rho \ge 1/2$, and
let $a\in\SG^{(\omega _1)}_{r,\rho}(\rr{2d})$
satisfy \eqref{eq:abSupport}.
Finally assume that 
$(\cB_1,\cC_1)$ is a weakly-I $\SG$-ordered pair with respect to
$(r,\rho ,\omega _1,\varphi ,\Omega )$.
Then,
\begin{equation}
\label{eq:wfincl}
\begin{aligned}
&\WFF^{m}_{\cC_1}(\op_\varphi(a)f) \subseteq \Lambda^m_{\cB_1}(f),
\\
&\Lambda^m_{\cB_1}(f)=\{(x,\xi)=\phi(y,\eta)
\colon (y,\eta)\in\WFF^{m}_{\cB_1}(f)\}^{\mathrm{con}_m},\quad
f \in \cS^\prime(\rr{d}),
\end{aligned}
\end{equation}
where $\phi$ is the canonical transformation generated by $\varphi$
in \eqref{eq:trsympl}.
\end{thm}

\par

\begin{thm}
\label{thm:3.23bis}
Let $\varphi \in \Phr$ be $m$-admissible, $m\in\{1,2,3\}$, and fulfill \eqref{eq:varphitau}
for a fixed $\tau \in (0,1)$.
Also let $\omega_2\in \mathscr P_{r,\rho}(\rr {2d})$, $r,\rho \ge 1/2$, and
let $b\in\SG^{(\omega_2)}_{r,\rho}(\rr{2d})$
satisfy \eqref{eq:abSupport}.
Finally assume that
$(\cB_2,\cC_2)$ is a weakly-II $\SG$-ordered pair with respect to
$(r,\rho ,\omega _2,\varphi ,\Omega )$
Then,
\begin{equation}
\label{eq:wfinclbis}
\begin{aligned}
&\WFF^{m}_{\cB_2}(\op_\varphi^*(b)f) \subseteq \Lambda^{m}_{\cC_2}(f)^*,
\\
&\Lambda^{m}_{\cC_2}(f)^*=\{(y,\eta)=\phi^{-1}(x,\xi)\colon(x,\xi)
\in\WFF^{m}_{\cC_2}(f)\}^{\mathrm{con}_m},\quad
f \in \cS^\prime(\rr{d}),
\end{aligned}
\end{equation}
where $\phi^{-1}$ is the inverse of the canonical transformation $\phi$
in \eqref{eq:trsympl}.
\end{thm}

\par

\begin{cor}
Let $\varphi \in \Phr$ be $m$-admissible, $m\in\{1,2,3\}$, and fulfill \eqref{eq:varphitau}
for a fixed $\tau \in (0,1)$.
Also let $\omega_1,\omega_2\in \mathscr P_{r,\rho}(\rr {2d})$, $r,\rho \ge 1/2$, and
$a\in\SG^{(\omega_1)}_{r,\rho}(\rr{2d})$ and $b\in\SG^{(\omega_2)}_{r,\rho}(\rr{2d})$
satisfy \eqref{eq:abSupport}.
Assume also that
$(\cB_1,\cC_1,\cB_2,\cC_2)$ are $\SG$-ordered with respect to
$$
r,\  \rho,\ \omega _1, \omega_2 ,\ \varphi \ \text{and}\ \Omega  .
$$
Then both \eqref{eq:wfincl} and \eqref{eq:wfinclbis} hold.
\end{cor}

\par

In the results above, $V^{\mathrm{con}_m}$ for $V\subseteq \Omega_m$,
is the smallest $m$-conical subset of $\Omega_m$ which includes $V$, $m\in\{1,2,3\}$.

\medskip

We prove only Theorem \ref{thm:3.23}. Theorem \ref{thm:3.23bis} follows by similar
arguments and is left for the reader.

\par

\begin{proof}[Proof of Theorem \ref{thm:3.23}]
	Since here we are dropping the ellipticity hypothesis on the amplitude $a$, we use 
	only the composition results between generalised 
	$SG$ pseudo-differential operators and Fourier integral
	operators established in Subsection \ref{subs:2.3}. That is, the proofs of the
	theorem again rely on the generalised $\SG$ asymptotic expansions 
	discussed in \cite{CoTo}, and on the properties of the admissible phase functions. 
		We now prove \eqref{eq:wfincl} in detail for the case $m=3$, by showing the opposite inclusion
		between the complements of the involved sets with respect to $\Omega_3$. 
		In the sequel, we write $\cB$ and $\cC$ in place of $\cB_1$ and $\cC_1$, respectively.
					
	\medskip	
		
		Let $(x_0,\xi_0)\notin\Lambda^3_\cB(f)$ for $f\in\cB$, and 
		set $2N=\min\{|x_0|,|\xi_0|\}$>0. 
		By its definition in \eqref{eq:wfincl},
		$\Lambda^3_{\cB}(f)$ is a closed $3$-conical set. Then, choosing $\varepsilon>0$ 
		sufficiently small, it is possible to find a $3$-conical set of the form
		\begin{align*}
			\Gamma_{3,x_0,\xi_0}^{4\varepsilon,4\varepsilon,N/4}
			&=\left\{
			(x,\xi)\in\rr{2d}\,\colon\,\left|\frac{x}{|x|}-\frac{x_0}{|x_0|}\right|<4\varepsilon,\right.
			\\
			&\hspace*{34mm}\left.\left|\frac{\xi}{|\xi|}-\frac{\xi_0}{|\xi_0|}\right|<4\varepsilon,
			|x|,|\xi|\ge \frac{N}{4}
			\right\}
		\end{align*}
		such that $\Gamma_{3,x_0,\xi_0}^{4\varepsilon,4\varepsilon,N/4}
		\cap\Lambda^3_{\cB}(f)=\emptyset$. Then, as it is also possible (see Subsection \ref{subs:4.1} above
		and \cite{Co}), pick $q\in\SG^{0,0}_{1,1,}$ such that 
		\[
			\supp q\subseteq\Gamma^{2\varepsilon,2\varepsilon,N/2}_{3,x_0,\xi_0}
			\text{ and }
			(x,\xi)\in\Gamma^{\varepsilon,\varepsilon,N}_{3,x_0,\xi_0}\Rightarrow
			q(x,\xi)=1.
		\]
		We now observe that $(y_0,\eta_0)=\phi^{-1}(x_0,\xi_0)\notin\WFF^3_\cB(f)$, in view
		of the definition of $\Lambda^3_\cB(f)$. Setting $2\widetilde{N}=\min\{|y_0|,|\eta_0|\}$,
		we can consider the subset of $\Omega_3$ given by
		\[
			W=\WFF^3_\cB(f)\cap\Omega_3^{\widetilde{N}}.
		\]
		$W$ is of course closed, and, by Remark \ref{rem:wfcon} it is $3$-conical.
		Then, there exist two $3$-conical neighborhoods $U,V$ of $W$ such that 
		$W\subset V\subset U\subset\Omega_3$. For instance, for an arbitrarily small $\tilde{\delta}>0$,
		one can consider the coverings of $W$ given by
		\[
			\widetilde{U}=\bigcup_{(z_0,\zeta_0)\in W}
			\Gamma^{4\tilde{\delta},4\tilde{\delta},\widetilde{N}/4}_{3,z_0,\zeta_0},
			\quad
			\widetilde{V}=\bigcup_{(z_0,\zeta_0)\in W}
			\Gamma^{2\tilde{\delta},2\tilde{\delta},\widetilde{N}/2}_{3,z_0,\zeta_0}.
		\]
		By a standard compactness argument on the unit sphere of $\rr{d}$, define $V$ and $U$
		as suitable finite subcoverings extracted from $\widetilde{V}$ and $\widetilde{U}$, respectively.
		Of course, since 
		$\Gamma^{2\tilde{\delta},2\tilde{\delta},\widetilde{N}/4}_{3,z_0,\zeta_0}
		\subset
		\Gamma^{4\tilde{\delta},4\tilde{\delta},\widetilde{N}/4}_{3,z_0,\zeta_0}$, we also get $W\subset V \subset U$,
		as desired. Then, take a symbol $\chi\in\SG^{0,0}_{1,1}$ such that
		\begin{align*}
			&\supp\chi\subset U,(y,\eta)\in V\Rightarrow \chi(y,\eta)=1,
			\\
			&
			\supp q\cap \{(x,\xi)=\phi(y,\eta)\;\colon (y,\eta)\in\supp\chi\}=\emptyset,
		\end{align*}
		All of the above is possible, choosing $\tilde{\delta}$ small enough, 
		in view of the hypotheses and of (3) in Definition \ref{def:gadm}. Indeed, we can start
		from a $3$-conical neighbourhood $Z\supset \Lambda^3_\cB(f)\cap\Omega_3^N$, obtained as a finite
		union of sets of the form $\Gamma_{3,t_0,\tau_0}^{2\varepsilon,2\varepsilon,N/2}$, 
		$(t_0,\tau_0)\in\Lambda^3_\cB(f)$,
		disjoint from $\Gamma^{4\varepsilon,4\varepsilon,N/4}_{3,x_0,\xi_0}$, by choosing 
		$\varepsilon>0$ suitably small. Observing that all
		the involved sets are $3$-conical, it is then possible to choose 
		$\tilde{\delta}$ small enough such that $\phi(U)\subset Z$, and $\chi$ with the desired properties.
		Let us now consider
		\begin{equation}\label{eq:qa1}
		\begin{aligned}
		 \hspace*{5mm}[\op(q)\circ\op_\varphi(a)](f)&=[\op(q)\circ\op_\varphi(a)\circ\op(1-\chi)]f
		\\
								&+[\op(q)\circ\op_\varphi(a)\circ\op(\chi)]f.
		\end{aligned}
		\end{equation}
		Recalling Remark \ref{rem:trivweight}, the weight $\vartheta_0,0(x,\xi)=1$ is invariant with
		respect to any $\SG$ diffeomorphism with $\SG^0$ parameter dependence. In the analysis of
		$C=\op(q)\circ\op_\varphi(a)\circ\op(\chi)$, we can then apply
		Theorems \ref{thm:0.1} and \ref{thm:3.1}. We find $C=\op_\varphi(c)$ with
		\begin{align*}
			c(x,\eta)\sim\sum 
			p_{\alpha\beta k l}(x,\eta)
			&\cdot(\partial^k_\xi q)(x,\varphi^\prime_x(x,\eta))
			\cdot(\partial^\alpha_\xi\partial^\beta_x a)(x,\eta)
			\\
			&\cdot(\partial^l_x\chi)(\varphi^\prime_\xi(x,\eta),\eta)\sim 0,
		\end{align*}
		which implies that $C\colon\cS^\prime\to\cS$. In fact, setting $\xi=\varphi^\prime_x(x,\eta)$,
		$y=\varphi^\prime_\xi(x,\eta)$, by \eqref{eq:trsympl} we have $(x,\xi)=\phi(y,\eta)$,
		and, by construction, $\supp\partial^k_\xi q$ $\cap\phi(\supp\partial^l_x\chi)=\emptyset$.
		Now, setting
		\[
			\Sigma=\{(y,\eta)\in\supp(1-\chi)\colon|y|,|\eta|\ge \widetilde{N}/2\},
		\]
		again by construction we have $\Sigma\cap\WFF^3_\cB(f)=\emptyset$. Then, there exist
		$p\in\SG^{0,0}_{1,1}$ such that $\op(p)f\in\cB$ and $p(x,\xi)\ge C >0$ on $\Sigma$, and
		$r,s\in\SG^{0,0}_{1,1}$ such that
		\begin{equation}\label{eq:sm}
			\op(r)-\op(s)\circ\op(p)\colon\cS^\prime\to\cS,
		\end{equation}
		with $r(x,\xi)\equiv1$ for $|x|,|\xi|\ge\frac{N}{2}$ belonging
		to a $3$-conical neighborhood of $\Sigma$. This can be proved by relying
		on the concept of $\SG$-ellipticity with respect to a symbol (or local md-ellipticity, cfr. \cite{Co}, Ch. 2, \S 3). 
		We can write
		\begin{align*}
			\op(1-\chi)f&=[\op(1-\chi)\circ\op(1-r)]f
			\\
			&+[\op(1-\chi)\circ[\op(r)-\op(s)\circ\op(p))]f
			\\
			&+[\op(1-\chi)\circ\op(s)][\op(p)f].
		\end{align*}
		The first term is in $\cS$, since the symbols of the two operators in the composition
		have, by construction, disjoint supports. The second term is in $\cS$, too,
		by \eqref{eq:sm}. The third term is in $\cB$, since this is true for $\op(p)f$,
		$\op(1-\chi)\circ\op(s)=\op(\lambda)$, with $\lambda\in\SG^{0,0}_{1,1}$,
		and $\cB$ is $\SG$-admissible. Then, by all the considerations above,
		the mapping properties of $\op_\varphi(a)$,
		the fact that also $\cC$ is $\SG$-admissible
		and that $q\in\SG^{0,0}_{1,1}$,
		\begin{align*}
			[\op(q)\circ\op_\varphi(a)]f&=[\op(q)\circ\op_\varphi(a)\circ\op(1-\chi)]f\mod\cS
			\\
			&=[\op(q)\circ\op_\varphi(a)]\underbrace{[\op(1-\chi)f]}_{\in\cB}\mod\cS
			\\
			&\Rightarrow [\op(q)\circ\op_\varphi(a)]f\in\cC,
		\end{align*}
		which proves $(x_0,\xi_0)\notin \WFF^3_\cC(\op_\varphi(a)f)$, and the claim.

	We observe that \eqref{eq:wfincl} for the case $m=1$ can be proved by the same argument used, e.g.,
	in \cite{Ku}, Ch. 10, \S 3. The case $m=2$ of \eqref{eq:wfincl} can then be obtained in a completely
	similar fashion, by exchanging the role of variable and covariable. The
	details are left for the reader.
\end{proof}

\begin{rem}
As it was observed in \cite{CJT2}, there is a simple and useful relation between the global 
wave-front set of $f$ and of $\widehat{f}$. Namely, with $m, n\in \{ 1,2,3 \}$ such that $n$ 
equals $2$, $1$ and $3$, when $m$ equals $1$, $2$ and $3$, respectively, we have
\[
	T \big (\WF{m}_{\cB}(f) \big ) = \WF{n}_{\cB_T}(\widehat f),
\]
where $\cB=M(\omega,\mathscr B)$, the torsion $T$ is given by
$T(x,\xi)=(-\xi,x)$, and $\mathscr B_T=\sets {F\circ T=T^*F}{F\in \mathscr B},$
$\omega_T=\omega\circ T$, $\cB_T=M(\omega_T,\mathscr B_T)$. Notice that
$\cF$ is bijective and continuous, together with its inverse, from $\cB_T$ onto $\cB$.

It is also
immediate to obtain a similar relation among the wave-front sets of $f$ and $\check{f}$,
where $\check{f}=f\circ R$ is the pull-back of $f$ under the action of the reflection $R(y)=-y$.
Indeed, since obviously, for any $a\in\SG^{m,\mu}_{r,\rho}$ and $f\in\cS^\prime$,
\[
	\op(a)\check{f}=[\op({\check{a}})f]\,\check{},
\]
it follows, for $m\in\{1,2,3\}$,
\[
	R(\WFF^m_\cB(f))=\WFF^m_{\check{\cB}}(\check{f}),
\]
where $\check{\mathscr B}=\sets {F\circ R=R^*F}{F\in \mathscr B}$,
$\check{\cB}=M(\omega_R,\check{\mathscr B})$, where
$\omega_R=\omega\circ R$. Notice that, in many cases, $\check{\cB}$=$\cB$:
for instance, this is true for all the functional spaces considered in Section \ref{sec3}, and, in general,
for all $M(\omega,\mathscr{B})$ such that $\check{\mathscr{B}}=\mathscr{B}$ and $\omega$ is even. 
Similarly to the above, $\check{}=R^*$ is bijective and continuous, together with its inverse, from 
$\cB$ onto $\check{\cB}$.

By \eqref{eq:typeI-II}, rewritten as
\[
	\op^*_{-\varphi^*}(a^*)f=(\cF\circ\op_\varphi(a)\circ\cF^{-1}f)\,{\check{}},
	\quad f\in\cS^\prime,
\]
and the above definitions of $\cB_T$ and $\check{\cB}$,
it also follow that, if $(\cB,\cC)$ is weakly-I $\SG$-ordered with
respect to $(r,\rho ,\omega _0,\varphi ,\Omega )$,
we find that, for any $a\in\SG^{(\omega_0)}_{r,\rho}$, supported
outside $\rr{d}\times\Omega$,
$\op^*_{-\varphi^*}(a^*)\colon\cB_T\rightarrow\check{\cC}_T$
continously, that is, $(\check{\cC}_T,\cB_T)$ is weakly-II $\SG$-ordered
with respect to $(r,\rho ,\omega _0,\varphi ,\Omega )$.
\end{rem}

\par
\subsection{Applications to $\SG$-hyperbolic problems}\label{subs:4.3}
In this subsection we apply the results obtained above to the
$\SG$-hyperbolic problems considered in \cite{CoMa,CoPa}, to which
we refer for the details omitted here.
We prove that, under natural conditions, the singularities described by the generalised
wave-front sets $\WFF^m_\cB(U_0)$, $m=1,2,3$, for a scalar or vector-valued
initial data $U_0\in\cB$, propagate to the solution $U(t)$, $t\in[-T,T]$,
in the sense that the points of $\WFF^m_\cC(U(t))$, $m=1,2,3$ lie on 
bicharacteristics curves determined by the phase functions
of the Fourier operators $A(t)$ such that, modulo smooth remainders,
$U(t)=A(t)U_0$. Here we choose, for simplicity, $\omega_0(x,\xi)=\norm{x}\norm{\xi}$.

Notice that the hyperbolic operators involved in such Cauchy problems 
arise naturally as local representations of (modified) wave operators of the form
$L=\Box_g-V$, with a suitable potential $V$ and the D'Alembert operator
$\Box_g$, on manifolds of the form $\rr{}_t\times M_x$, equipped with a hyperbolic
metric $g=\mathrm{diag}(-1,h)$, where $h$ is a suitable Riemannian
metric on the manifold with ends $M$. In this way,
\[
	L=\Box_g-V=-\partial_t^2+\Delta_h-V=-\partial_t^2+P,
\] 
where $\Delta_h$ is the Laplace-Beltrami operator on $M$ associated with the metric $h$
and we have set $P=\Delta_h-V$.
In the following Example \ref{ex:hypmf}, we show that this indeed occurs,
considering a rather simple situation with $\dim M=2$. 

\begin{example}\label{ex:hypmf}
Assume $\dim M=2$ and consider, as local model of one of the ``ends'' of $M$, the cylinder in $\rr{3}$
given by $u^2+v^2=1$, $z>1$, that is the manifold $\mathcal{C}=S^1\times(1,+\infty)$.
First, we have to equip $\ciliuno$ with a $\esse$-structure, namely, a $\SG$-compatible atlas
(see \cite{Co, Sc}). This can be easily accomplished here, by choosing a standard product atlas
on $S^1\times(1,+\infty)$, identifying $S^1$ with the unit circle in $\rr{2}$ centred at the
origin, as we now explain. With coordinates $(u,v)$ on $\rr{2}$, set
\begin{equation*}
\Omega'_1 := S^1 \backslash \{(0,1)\}, \,\,\,\,\,\,\, \Omega'_2 := S^1 \backslash \{(0,-1)\},
\end{equation*}
\begin{align*}
&\nu'_1 : \Omega'_1  \rightarrow \rr{}\colon(u,v)  \mapsto \frac{u}{1-v},
\\
&\nu'_2 : \Omega'_2  \rightarrow \rr{}\colon(u,v) \mapsto \frac{u}{1+v}
\end{align*}
It is immediate to show that $(\nu'_1)^{-1} : \nu'_1(\Omega'_1)  \rightarrow  \Omega'_1 \subset S^1$ is
\begin{equation*}
t \mapsto \left(\frac{2t}{1+t^2}, -\frac{1-t^2}{1+t^2} \right),
\end{equation*}
and $(\nu'_2)^{-1} : \nu'_2(\Omega'_2)  \rightarrow  \Omega'_2 \subset S^1$ is
\begin{equation*}
t \mapsto \left(\frac{2t}{1+t^2}, \frac{1-t^2}{1+t^2} \right),
\end{equation*}
so that, for $\displaystyle t \in \nu'_2(\Omega'_1 \bigcap \Omega'_2)=(-\infty,0)\cup(0,+\infty)$, we find
\begin{equation*}
\begin{split}
\nu_{12}^\prime(t)  = \nu'_1((\nu'_2)^{-1}(t)) = \nu'_1 \left(\frac{2t}{1+t^2}, \frac{1-t^2}{1+t^2} \right) = \frac{1}{t}. 
\end{split}
\end{equation*}
Now set
\begin{equation*}
\Omega_1 := \Omega'_1 \times (1,+\infty), \,\,\,\,\,\,\, \Omega_2 := \Omega'_2 \times (1,+\infty),
\end{equation*}
define $\nu_1 : \Omega_1  \rightarrow  U_1 \subset \rr{2}$ by
\begin{equation*}
(u,v,z)  \mapsto (\nu'_1(u,v) , 1) \frac{z}{\sqrt{1 + (\nu'_1(u,v))^2}} = \left(\frac{u}{1-v} , 1\right) \frac{z}{\sqrt{1 + \dfrac{u^2}{(1-v)^2}}},
\end{equation*}
and $\nu_2 : \Omega_2  \rightarrow  U_2 \subset \rr{2}$ by
\begin{equation*}
(u,v,z)  \mapsto (\nu'_2(u,v) , 1) \frac{z}{\sqrt{1 + (\nu'_2(u,v))^2}} = \left(\frac{u}{1+v} , 1\right) \frac{z}{\sqrt{1 + \dfrac{u^2}{(1+v)^2}}}.
\end{equation*}
\noindent Again, it is easy to obtain the expressions of $\nu_1^{-1} : U_1  \rightarrow  \Omega_1$
and $\nu_2^{-1} : U_2  \rightarrow  \Omega_2$, and to prove that, with coordinates $x=(x_1,x_2)$ on $\rr{2}$,
\begin{equation*}
\begin{split}
\nu_{12}(x_1,x_2)  &  = \nu_1((\nu_2)^{-1}(x_1,x_2))
 = \left(\frac{x_2}{x_1},1\right) x_1 =(x_2 , x_1),
\end{split}
\end{equation*}
which shows that the atlas $\{(\Omega_j, \nu_j), j=1,2\}$ defines a $\esse$-structure on $\ciliuno$,
since $\norm{\nu_{12}(x)}=\norm{x}$ (see again, e.g., \cite{Co,Sc}).
%
%
Next, for any $\mu>0$, define a metric $h'$ on $\{(u,v,z) \in \rr{3} : z > 1\}$ by
\begin{equation*}
(h'_{ij}) := \left(
\begin{array}{ccc}
\displaystyle \frac{z^2}{4 \semi{z}^\mu} & 0 & 0 \\
0 & \displaystyle \frac{z^2}{4 \semi{z}^\mu} & 0 \\
0 & 0 & \displaystyle\frac{1}{ \semi{z}^\mu}
\end{array} \right). 
\end{equation*}
With $x \in U_1$, and denoting by $J_1$ the Jacobian matrix of $\nu_1^{-1}$,
it turns out that the pull-back metric $h := (\nu_1^{-1})^* h'$ on $\ciliuno$ is given by
\begin{equation*}
(h_{ij}) =  J_1 \, ((\nu_1^{-1})^*h'_{ij})|_{U_1}\, J_1^t = \left(
\begin{array}{cc}
\displaystyle \frac{1}{ \semi{x}^\mu} & 0  \\
0 & \displaystyle \frac{1}{ \semi{x}^\mu} \\
\end{array} \right).
\end{equation*}
In the same way, one can show that the metric $h$ has the same local expression for $x \in U_2$. Finally, let us compute the Laplace-Beltrami operator on $\ciliuno$ associated with $h$ in the
chosen local coordinates. We have, of course, $(h^{ij}) = \mathrm{diag}(\semi{x}^\mu,\semi{x}^\mu)$
and $\sqrt{|h|} = \semi{x}^{-\mu} $, thus, for any $f \in C^\infty(\ciliuno)$,
\begin{equation*}
\begin{split}
\Delta_h f & = \frac{1}{\semi{x}^{-\mu}} \sum_{i,j = 1}^2 \frac{\partial}{\partial x^j} \left( \semi{x}^{-\mu} h^{ij} \frac{\partial f}{\partial x^i}\right) \\
& =  \semi{x}^{\mu} \sum_{i,j=1}^2  \frac{\partial}{\partial x^j} \left( \delta^{ij} \frac{\partial f}{\partial x^i}\right) \\
& = \semi{x}^{\mu} \sum_{i=1}^2  \frac{\partial^2 f}{\partial x_i^2} =  \semi{x}^{\mu} \left( \frac{\partial^2}{\partial x_1^2} + \frac{\partial^2}{\partial x_2^2}\right) f,
\end{split}
\end{equation*}
that is
\begin{equation*}
\Delta_h = \norm{x}^\mu \Delta,
\end{equation*}
where $\Delta$ is the standard Laplacian on $\rr{2}$. Choosing $V(x)=\norm{x}^\mu$, 
the local symbol of $P=\Delta_h-V$ is
\begin{equation}
\label{eq:symdelta}
p(x,\xi) = - \norm{x}^\mu \norm{\xi}^2 = -(1+x_1^2+x_2^2)^{\frac{\mu}{2}} (1+\xi_1^2 + \xi_2^2),
\end{equation}
which obviously belongs to $\SG^{2,\mu}_{1,1}(\rr{2}\times\rr{2})$ and is $\SG$-elliptic.
In \cite{CoMa2}, the spectral theory for elliptic self-adjoint operators, generated by local symbols with 
(different) orders $m,\mu>0$, has been considered. On the other hand,
the case $\mu=2$ is of special interest in the context of the $\SG$-hyperbolic operators (see below),
since then we have that $L$, in local coordinates, is given by
\[
	L=\Box_g-V=-\partial_t^2+\Delta_h-V=-\partial_t^2-\norm{x}^2(1-\Delta)=D_t^2-\norm{x}^2\norm{D_x}^2.
\]
\end{example}

In the sequel of this subsection, the subscript ``$\mathrm{cl}$'' denotes the subclasses of
$\SG$ symbols which are classical, see \cite{CoPa}. Notice that the symbol \eqref{eq:symdelta} 
actually belongs to $\SG^{2,\mu}_{1,1,\mathrm{cl}}(\rr{2}\times\rr{2})$.
We first need to recall some definitions and results, mainly taken from
\cite{coriasco, coriasco2, CoPa}.

\begin{defn}
    \label{def:4.13}
    Let $J=[-T,T] \subset \rr{}$, $T > 0$, and consider the linear operator
    \begin{equation}
	\label{eq:4bis.69}
	L = D_{t}^\nu + P_{1}(t)\,D_{t}^{\nu-1} + \dots + P_{\nu}(t),
    \end{equation}
    with $P_j(t)=\Op{p_j(t)}$,
    $p_{j} = p_{j}(t,x,\xi) \in C^\infty(J,\SG^{1,1}_{1,1,\mathrm{cl}}(\rr{2d}))$. Let
    \[
    l(x,\xi,t,\tau) = \tau^\nu + q_{1}(t,x,\xi)\tau^{\nu-1} + \dots + q_{\nu}(t,x,\xi)
	\]
	 be the principal symbol of $L$, with
	$q_{j} \in C^\infty(J,\SG^{j,j}_{1,1,\mathrm{cl}}(\rr{2d}))$
	such that $q_{j}(t,.,..)$ is the principal symbol of $p_j(t,.,..)$, in the sense of
		 \eqref{eq:classsymp}.
    $L$ is called $\SG$-classical hyperbolic with constant multiplicities if
    the characteristic equation
    \begin{equation}
	\label{eq:4bis.73}
	\tau^\nu + q_{1}(t,x,\xi)\tau^{\nu-1} + \dots + q_{\nu}(t,x,\xi) = 0
    \end{equation}
    has $\mu \le \nu$ distinct real roots $\tau_{j} = \tau_{j}(t,x,\xi)
    \in C^\infty(J,\SG^{1,1}_{1,1,\mathrm{cl}}(\rr{2d}))$ with multiplicities $l_{j}$, $1\le l_j\le \nu$,
    $j=1,\dots,\mu$, which satisfy, for a suitable $C > 0$ and all $t\in J$, $x,\xi\in\rr{d}$,
    \begin{equation}
	\label{eq:4bis.2}
        \tau_{j+1}(t,x,\xi) - \tau_{j}(t,x,\xi) \ge
	C \norm{\xi} \norm{x}, \mbox{  } j = 1, \dots, \mu - 1.
    \end{equation}
    $L$ is called strictly hyperbolic if it is hyperbolic with constant
    multiplicities and the multiplicity of all the $\tau_{j}$, $j=1,\dots,\nu=\mu$, is equal to 1.
\end{defn}
A standard strategy to solve the Cauchy problem
\begin{equation}
    \label{eq:4bis.74}
    \left\{
     \begin{array}{ll}
	 L u(t) = 0,          & t \in J,
	 \\
	 D_{t}^k u(0) = u^k_{0},     & k=0,\dots,\nu-1,
     \end{array}
    \right.
\end{equation}
for $L$ hyperbolic with constant multiplicities
and initial data $u_{0}^k$, $k=1,\dots,\nu-1$, chosen in appropriate functional spaces,
is to show that this is equivalent to solving, modulo smooth elements, 
a Cauchy problem for a first order system
\[
    \left\{
    \begin{array}{l}
      \dfrac{\partial U}{\partial t}(t) - iK(t)\,U(t)  = 0,
      \mbox{  } t \in J,\\[.3cm]
      U(0) = U_{0},
    \end{array}
    \right.
\]
with a \emph{coefficient matrix} $K$ of special form. In our case, one obtains that
$K = \Op{(k_{ij}(t,x,D))_{i,j}}$, is a $\mu\nu \times \mu\nu$ matrix of $\SG$ pseudo-differential
operators with symbols $k_{ij} \in C^\infty(J, \SG^{1,1}_{1,1,\mathrm{cl}})$.
Under suitable assumptions, see \cite{CoMa, CoPa},
the principal part $k_{1}$ of $k = k_{1} + k_{0}$, $k_{j} \in
C^\infty(J,\SG^{j,j}_{1,1,\mathrm{cl}})$, $j=0,1$, turns out to be diagonal, so that the
system will be symmetric, cfr. \cite{Co, coriasco, coriasco2}. This implies that
the corresponding Cauchy problem is well-posed.
One of the main advantages for using
this algorithm is the following Proposition \ref{prop:4bis.2},
which is an adapted version
of the Mizohata Lemma of Perfect Factorization, proved in
\cite{CoRo} for the general $\SG$ symbols (see also the references quoted therein).

\begin{prop}
    \label{prop:4bis.2}
    Let $L$ be a $\SG$-classical
    hyperbolic linear operator with constant multiplicities
    $l_{j}$, $j=1,\dots,\mu \le \nu$, as in Definition
    \ref{def:4.13}. Then, it is possible to factor $L$ as
    \[
	L = L_{\mu} \cdots L_{1} + \sum_{s=1}^\nu \Op{r_{s}(t)} \, D_{t}^{\nu-s}
    \]
    with $L_{j}= (D_{t} - \Op{\tau_j(t)})^{l_{j}} + \sum_{k=1}^{l_{j}}
           \Op{s_{jk}(t)} \, (D_{t} - \Op{\tau_j(t)})^{l_{j}-k}$ and
    \begin{align*}
	&
	s_{jk} \in C^\infty(J, \SG^{k-1,k-1}_{1,1,\mathrm{cl}}(\rr{2d})),
	r_{s} \in C^\infty(J,\cS(\rr{2d})),
	\\
	&
	j=1, \dots, \mu, k = 1, \dots, l_{j}, s= 1, \dots, \nu.
    \end{align*}
\end{prop}
The following corollary, also obtained in \cite{CoRo}, follows by means of a reordering
of the roots $\tau_{j}$ of the principle symbol of $L$.

\begin{cor}
    \label{cor:4bis.3}
    Let $c_{j}$, $j=1, \dots, \mu$, denote the reorderings of the
    $\mu$-tuple $(1, \dots,$ $\mu)$ given by
    \begin{eqnarray*}
	& &
	c_{j}(i) =
	\left\{
	\begin{array}{ll}
	    j + i       & \mbox{for $j + i \le \mu$}
	    \\
	    j + i - \mu & \mbox{for $j + i  >  \mu$},
	\end{array}
	\right.
	\\
	& &
	i, j = 1, \dots, \mu,
    \end{eqnarray*}
    that is, $c_{1} = (2, \dots, \mu, 1)$, \dots,
    $c_{\mu} = (1, \dots, \mu)$. Then, under the same hypotheses of
    Proposition \ref{prop:4bis.2}, we have
    \[
    L = L^{(m)}_{c_{m}(\mu)} \cdots L^{(m)}_{c_{m}(1)} +
	\sum_{s=1}^\nu \Op{r^{(m)}_{s}(t)} D_{t}^{\nu-s}
    \]
    with $L^{(m)}_{j}= (D_{t} - \Op{\tau_j(t)})^{l_{j}} + \sum_{k=1}^{l_{j}}
             \Op{s^{(m)}_{jk}(t)} \, (D_{t} - \Op{\tau_j(t)})^{l_{j}-k}$ and
    \begin{eqnarray*}
	& &
	s^{(m)}_{jk} \in C^\infty(J, \SG^{k-1, k-1}_{1,1,\mathrm{cl}}(\rr{2d})),
	r^{(m)}_{s} \in C^\infty(J,\cS(\rr{2d})),
	\\
	& &
	m,j=1, \dots, \mu, k = 1, \dots, l_{j}, s= 1, \dots, \nu.
    \end{eqnarray*}
\end{cor}
\begin{defn}
We say that a $\SG$-classical hyperbolic operator 
$L$ is of Levi type if it satisfies the $\SG$-Levi
condition\footnote{Let us observe that \eqref{eq:4bis.111} needs to be
fulfilled only for a single value of $m$.}
\begin{equation}
    \label{eq:4bis.111}
    s^{(m)}_{jk} \in C^\infty(J, \SG^{0,0}_{1,1,\mathrm{cl}}(\rr{2d})),
    \hspace{0.2cm}
    m,j=1, \dots, \mu,
    k=1, \dots, l_{j}.
\end{equation}
\end{defn}
Theorem \ref{thm:4bis.1} below gives the well-posedness for the Cauchy problem
\eqref{eq:4bis.74} and the propagation results of the global wave-front sets of modulation
space type $\WFF^m_\cC(u(t))$, $m=1,2,3$, for the corresponding solution $u(t)$, 
under natural assumptions on the initial data and the modulation space $\cC$.
It immediately follows by the analysis of $\SG$-classical hyperbolic Cauchy problems in 
\cite{CoPa}, by Section \ref{sec3} and by Theorem \ref{thm:3.22}.

\par

    We here consider a $\SG$-classical hyperbolic operator $L$ with constant multiplicities and
    of Levi type, and denote by $l = \max \{ l_{1}, \dots, l_{\mu} \}$
    the maximum multiplicity of the distinct real roots $\tau_{j}$, $j=1,\dots,\mu$, of the
    characteristic equation \eqref{eq:4bis.73}.
    Then, as proved in \cite{coriasco2,CoPa},
 for any choice of initial data $u^j_0\in\cS^\prime(\rr{d})$, $j=0,\dots,\nu-1$,
    the Cauchy problem \eqref{eq:4bis.74} admits a unique solution $u
    \in C(J^\prime,\cS^\prime(\rr{d}))$, $J^\prime=[-T^\prime,T^\prime]$,
    $0<T^\prime\le T$. Collecting the initial conditions in the vector
    \begin{eqnarray*}
	& &
	c_0 = \left(
	\begin{array}{c}
	    u_{0}^{0}
	    \\
	    u^{1}_{0}\rule{0mm}{5mm}
	    \\
	    \vdots\rule{0mm}{5mm}
	    \\
	    u^{\nu-1}_{0}\rule{0mm}{5mm}
	\end{array}
	\right),
    \end{eqnarray*}
    the solution $u$ is given by
    \[
	u(t) = (A_1(t) + \dots + A_{\mu}(t)) c_0,
    \]
    where each $A_{j}(t) = \op_{\varphi_{j}(t)}(a_{j}(t))$
    is a type I FIO with regular phase function $\varphi_{j}
    \in C^\infty(J^\prime, \Phr)\cap C^\infty(J^\prime,\SG^{1,1}_{1,1,\mathrm{cl}}(\rr{2d}))$,
    solution of the eikonal equation associated with $\tau_{j}$,
    and vector-valued amplitude functions $a_{j} = (a_{j0}, \dots, a_{j\nu-1})$
    with $a_{jk} \in
    C^\infty(J^\prime,\SG^{l-k-1,l-k-1}_{1,1,\mathrm{cl}}(\rr{2d}))$,
    $j = 1, \dots, \mu$, $k=0, \dots, \nu-1$.

\begin{thm}
    \label{thm:4bis.1}
	Let $L$ be as above, and let $u_0^k\in\cB_k$, $k=0,\dots,\nu-1$, with the
	$\nu$-tuple of modulation spaces
    $(\cB_0,\dots,\cB_{\nu-1})$.
    Also assume that the modulation space $\cC$ is such that $(\cB_k,\cC)$, $k=0,\dots,\nu-1$, 
    are weakly $\SG$-ordered with respect to
    \[
    	1, 1, \norm{x}^{l-k-1} \norm{\xi}^{l-k-1}, \varphi_k(t) \text{ and } \emptyset.
    \]
    Then the Cauchy problem \eqref{eq:4bis.74} is well-posed
    with respect to 
    $(\cB_0,\dots,\cB_{\nu-1})$ and $\cC$, $u\in C(J^\prime,\cC)$, and
    \[
    \WFF^m_\cC( u(t) ) \subseteq \left[ \bigcup_{j=1}^\mu \bigcup_{k=0}^{\nu-1}
	                 \phi_{j}(t)(\WFF^m_{\cB_k} (u_{0}^k)) \right], \hspace{5mm} m=1,2,3,
     \]
    where $\phi_j(t)$ is the canonical transformation \eqref{eq:trsympl} associated
    with the phase function $\varphi_{j}(t)$.
\end{thm}
\begin{cor}
Assume that the hypotheses of Theorem \ref{thm:4bis.1} hold. Then
$\WFF^m_\cC(u(t))$, $t\in J^\prime$, $m=1,2,3$, consists of arcs of bicharacteristics, 
generated by the phase functions $\varphi_j(t)$ and emanating from points 
belonging to  $\WFF^m_{\cB_k} (u_{0}^k)$, $k=0,\dots,\nu-1$.
\end{cor}




\begin{thebibliography}{150}

\bibitem{AsFu78}
K.~Asada and D.~Fujiwara.
\emph{On Some Oscillatory Transformation in $L^2(\rr{n})$},
Japan J. Math., \textbf{4} (1978), 229--361.

\bibitem{Berger}
M.~S. Berger.
\newblock \emph{Nonlinearity and Functional Analysis}.
\newblock Academic Press, 1977.

\bibitem{BoBuRo} P. Boggiatto, E. Buzano, L. Rodino \emph{Global
Hypoellipticity and Spectral Theory},  Mathematical Research, 92,
Akademie Verlag, Berlin, 1996.

\bibitem{Bn1} {J. M. Bony} \emph{Caract{\'e}risations des
Op{\'e}rateurs Pseudo-Diff{\'e}rentiels  \rm{in: S{\'e}\-mi\-nai\-re
sur les {\'E}quations aux D{\'e}riv{\'e}es Partielles}}, 1996--1997,
Exp. No. XXIII, S{\'e}min. {\'E}cole Polytech., Palaiseau, 1997. 

\bibitem{Bn2} {J. M. Bony} \emph{Sur l'In\'egalit\'e de
Fefferman-Phong \rm{in: S\'eminaire sur les  \'Equations aux
D\'eriv\'ees Partielles}}, 1998--1999, Exp. No. III, S\'emin. \'Ecole
Polytech., Palaiseau, 1999. 

\bibitem{BC} {J. M. Bony, J. Y. Chemin} \emph{Espaces Functionnels
Associ\'es au Calcul de Weyl-H{\"o}rmander}, Bull. Soc. math. France
\textbf{122} (1994), 77--118. 

\bibitem{BoL} {J. M. Bony, N. Lerner,} \emph{Quantification
Asymptotique et Microlocalisations d'Ordre Su\-p\'e\-rieur I},
Ann. Scient. \'Ec. Norm. Sup., \textbf{22} (1989, 377--433. 

\bibitem{Bo11}
M.~Borsero.
\newblock Microlocal Analysis and Spectral Theory
of Elliptic Operators on Non-compact Manifolds.
\newblock Tesi di Laurea Magistrale in Matematica, Universit{\'a}
di Torino, 2011.

\bibitem{BuNi} {E. Buzano, N. Nicola,} \emph{Pseudo-differential
Operators and Schatten-von Neumann Classes, \rm{in: P. Boggiatto,
R. Ashino, M. W. Wong (eds),} Advances in Pseudo-Differential
Operators, Proceedings of the Fourth ISAAC Congress,} Operator
Theory: Advances and Applications, Birkh{\"a}user Verlag, Basel, 
(2004).

\bibitem{BuTo} E. Buzano, J. Toft \emph{Schatten-von Neumann properties in
the Weyl calculus}, J. Funct. Anal. \textbf{259} (2010), 3080--3114. 

\bibitem{CorNicRod1} E. Cordero, F. Nicola, L. Rodino \emph{On the Global
Boundedness of Fourier Integral Operators}, Ann. Global Anal. Geom.
\textbf{38} (2010), 373--398.

\bibitem{Co} H. O. Cordes \emph{The Technique of Pseudodifferential
Operators.} Cambridge Univ. Press, 1995.

\bibitem{coriasco} S. Coriasco \emph{Fourier Integral Operators in $\SG$
classes I. Composition Theorems and Action on $\SG$ Sobolev spaces}.
Rend. Sem. Mat. Univ. Pol. Torino, 1998.

\bibitem{coriasco2} S. Coriasco \emph{Fourier Integral Operators
in SG Classes II: Application to SG Hyperbolic Cauchy Problems}.
Ann. Univ. Ferrara, Sez. VII -- Sc. Mat. \textbf{44} (1998), 81--122.

\bibitem{CJT1} S. Coriasco, K. Johansson, J. Toft \emph{Local
wave-front sets of Banach and Fr{\'e}chet types, and pseudo-differential
operators}, {Monatsh. Math.},
\textbf{169} (2013) 285--316.
 
\bibitem{CJT2} S. Coriasco, K. Johansson, J. Toft \emph{Global wave-front
sets of Banach, Fr{\'e}chet and Modulation space types, and pseudo-differential
operators}, {J. Differential Equations} (2013), \textbf{254} (2013), 3228--3258.

\bibitem{CJT3} S. Coriasco, K. Johansson, J. Toft
\emph{Global Wave-front Sets of Intersection and Union Type
{\rm
{in: M. Ruzhansky, V. Turunen (eds)}}
Fourier Analysis - Pseudo-differential Operators, Time-Frequency
Analysis and Partial Differential Equations,}
Trends in Mathematics, Birkh{\"a}user,
Heidelberg NewYork Dordrecht London, 2014, pp. 91--106.

\bibitem{CoMa} S. Coriasco, L. Maniccia \emph{Wave front set
at infinity and hyperbolic linear operators with multiple characteristics},
{Ann. Global Anal. Geom.}, \textbf{24} (2003), 375--400.

\bibitem{CoMa2} S. Coriasco, L. Maniccia \emph{Wave front set
at infinity and hyperbolic linear operators with multiple characteristics},
{Ann. Global Anal. Geom.}, \textbf{24} (2003), 375--400.

\bibitem{CoPa} S. Coriasco, P. Panarese \emph{Fourier Integral Operators Defined by 
Classical Symbols with Exit Behaviour}. {Math. Nachr.}, \textbf{242} (2002), 61-78.

\bibitem{CoRo} S. Coriasco, L. Rodino \emph{Cauchy problem for
$\SG$-hyperbolic equations with constant multiplicities}. Ric. di
Matematica, \textbf{48} (Suppl.) (1999), 25--43.

\bibitem{CoRu} S. Coriasco, M. Ruzhansky \emph{Global $L^p$-continuity
of Fourier Integral Operators}. 
Trans. Amer. Math. Soc. (2014) DOI 10.1090/S0002-9947-2014-05911-4.

\bibitem{CoSch} S. Coriasco, R. Schulz \emph{Global Wave Front Set of
Tempered Oscillatory Integrals with Inhomogeneous Phase Functions}.
J. Fourier Anal. Appl.
(2013), DOI 10.1007/s00041-013-9283-4.

\bibitem{CoTo} S. Coriasco, J. Toft \emph{Asymptotic expansions
for H{\"o}rmander symbol classes in the calculus of pseudo-differential
operators}. J. Pseudo-Differ. Oper. Appl. 
DOI 10.1007/s11868-013-0086-9 (2013).

\bibitem{CoTo2} S. Coriasco, J. Toft \emph{A calculus of Fourier integral operators
with inhomogeneous phase functions on $\rr{d}$}. Preprint arxiv: (2014).

\bibitem{F1}  H.~G.~Feichtinger \emph{Modulation spaces on locally
compact abelian groups. Technical report}, {University of
Vienna}, Vienna, 1983; also in: M. Krishna, R. Radha,
S. Thangavelu (Eds) Wavelets and their applications, Allied
Publishers Private Limited, NewDehli Mumbai Kolkata Chennai Hagpur
Ahmedabad Bangalore Hyderbad Lucknow, 2003, pp. 99--140.

\bibitem{Feichtinger3}  {H. G. Feichtinger and K. H. Gr{\"o}chenig}
\emph{Banach spaces related to integrable group representations and
their atomic decompositions, I}, J. Funct. Anal., \textbf{86}
(1989), 307--340.

\bibitem{Feichtinger6} {H. G. Feichtinger and K. H. Gr{\"o}chenig}
\emph{Modulation spaces: Looking back and ahead},
Sampl. Theory Signal Image Process. \textbf{5} (2006), 109--140.

\bibitem{Fo}  {G. B. Folland} \emph
{Harmonic analysis in phase space}, {Princeton U. P., Princeton},
1989.

\bibitem{Gro-book} K. Gr\"{o}chenig \newblock \emph{Foundations of
Time-Frequency Analysis},
\newblock Birkh\"auser, Boston, 2001.

\bibitem{GT} K. Gr{\"o}chenig, J. Toft \emph{Isomorphism properties of
Toeplitz operators and pseudo-differential operators between modulation
spaces}, J. Anal. Math. \textbf{114} (2011), 255--283.

\bibitem{Ho1} L. H\"ormander \emph{The Analysis of Linear
Partial Differential Operators}, vol {I--IV}.
Springer-Verlag, Berlin Heidelberg NewYork Tokyo, 1983, 1985.

\bibitem{Ku} H.~Kumano-go.
\emph{Pseudo-Differential Operators}.
\newblock MIT Press, 1981.

\bibitem{LuRa} F. Luef, Z. Rahbani \emph{On pseudodifferential operators with
symbols in generalized Shubin classes and an application to Landay-Weyl operators},
Banach J. Math. Anal. \textbf{5} (2011), 59--72.

\bibitem{Me} R. Melrose \emph{Geometric scattering theory.} Stanford Lectures.
Cambridge University Press, Cambridge, 1995.

%
\bibitem{RuSu} M. Ruzhansky, M. Sugimoto,  
      \emph{Global $L^2$ boundedness 
      theorems for a class of Fourier integral operators.} 
      Comm. Partial Differential Equations
      \textbf{31} (2006), 547--569. 

%
\bibitem{Sc} E. Schrohe \emph{Spaces of weighted symbols and
weighted Sobolev spaces on manifolds} In: H. O. Cordes,
B. Gramsch, and H. Widom (eds), Proceedings, Oberwolfach,
\textbf{1256} Springer LMN, New York, 1986, pp. 360-377.
%

\bibitem{To8} J. Toft \emph{Continuity
properties for modulation spaces with applications to
pseudo-differential calculus, II}, {Ann. Global Anal. Geom.},
\textbf{26} (2004), 73--106.

\bibitem{To10} J. Toft \emph{Schatten-von Neumann properties in the
Weyl calculus, and calculus of metrics on symplectic vector spaces},
Ann. Glob. Anal. and Geom. \textbf{30} (2006), 169--209.

\end{thebibliography}
\end{document}